\documentclass[final]{siamltex}
\usepackage{amsfonts}

\usepackage{graphics}

\usepackage{amssymb} \usepackage{color}
\usepackage{graphicx}

\title{Error estimates and a two grid scheme for approximating transmission eigenvalues\thanks{This
 work was supported by the National Science Foundation of China (Grant Nos.11561014, 11201093).}}

\author{YIDU YANG$^{\small \dag}$, JIAYU HAN$^{\small \dag}$, AND HAI BI
\thanks{\tt School of Mathematics and Computer Science,Guizhou
Normal University,GuiYang,550001, China ({\tt
ydyang@gznu.edu.cn,hanjiayu126@126.com,bihaimath@gznu.edu.cn}).}}
\begin{document}
\maketitle
\begin{abstract}
In this paper, using the linearization technique we write the
Helmholtz transmission eigenvalue problem as an equivalent
nonselfadjoint linear eigenvalue problem whose  left-hand side term
is a selfadjoint, continuous and coercive sesquilinear form. To
solve the resulting nonselfadjoint eigenvalue problem, we give an
$H^{2}$ conforming finite element discretization and establish a two
grid discretization scheme. We present a complete  error analysis
for both discretization schemes, and theoretical analysis and
numerical experiments show that the methods presented in this paper
can efficiently compute real and complex transmission eigenvalues.
\end{abstract}

\pagestyle{myheadings} \thispagestyle{plain} \markboth{YIDU YANG,
JIAYU HAN AND HAI BI }{ ERROR ESTIMATES FOR TRANSMISSION
EIGENVALUES}

\section{Introduction}

\indent Transmission eigenvalues not only have wide physical
applications, for example, they can be used to obtain estimates for
the material properties of the scattering object
\cite{cakoni1,cakoni2}, but also have theoretical importance in the
uniqueness and reconstruction in inverse scattering theory
\cite{colton1}. Many papers such as
\cite{cakoni2,colton1,kirsch,paivarinta,rynne} study the existence
of transmission eigenvalues, and \cite{cakoni2,cakoni3,colton3}
explore  upper and lower bounds for the index of refraction
$n(x)$ from  knowledge of the transmission eigenvalues.
The attention from the computational mathematics community is also increasing (see, e.g., \cite{an1,cakoni4,colton2,gintides,ji,ji2,monk,sun1,su4}). \\
\indent In this paper, we consider the Helmholtz transmission
eigenvalue problem: Find $k\in \mathbb{C}$, $w, \sigma\in L^{2}(D)$,
$w-\sigma\in H^{2}(D)$ such that
\begin{eqnarray}\label{s1.1}
&&\Delta w+k^{2}n(x)w=0,~~~in~ D,\\\label{s1.2}
 &&\Delta
\sigma+k^{2}\sigma=0,~~~in~ D,\\\label{s1.3}
 &&w-\sigma=0,~~~on~ \partial
D,\\\label{s1.4}
 &&\frac{\partial w}{\partial \nu}-\frac{\partial
\sigma}{\partial \nu}=0,~~~ on~\partial D,
\end{eqnarray}
where $D\subset \mathbb{R}^{d}$ $(d=2,3)$ is a bounded simply
connected set containing an inhomogeneous medium, and $\nu$ is the unit outward normal
to $\partial D$.\\
\indent From \cite{cakoni3,rynne} we know that for $u=w-\sigma\in
H_{0}^{2}(D)$,  the weak formulation for the transmission eigenvalue
problem (\ref{s1.1})-(\ref{s1.4}) can be stated as follows: Find
$k^{2}\in \mathbb{C}$, $k^{2}\not=0$, $u\in
H_{0}^{2}(D)\backslash\{0\}$ such that
\begin{eqnarray}\label{s1.5}
(\frac{1}{n(x)-1}(\Delta u+k^{2} u),\Delta v+\overline{k^{2}} n(x)v
)_{0}=0,~~~\forall v \in~H_{0}^{2}(D),
\end{eqnarray}
where  $(\psi,\varphi)_{0}=\int_{D}\psi\overline{\varphi}dx$ denotes
the $L^{2}(D)$ inner product. We denote $\lambda=k^{2}$ as usual,
then (\ref{s1.5}) is a quadratic
eigenvalue problem.\\
\indent In recent years, based on this weak formulation, various
numerical methods have been presented to solve the transmission
eigenvalue problem. The first numerical treatment appeared in
\cite{colton2} where three finite element methods were proposed for
the Helmholtz transmission eigenvalues which have been further
developed in \cite{an1,cakoni4,gintides,ji,ji2,sun1}. Among them
\cite{colton2,ji2,sun1} studied the $H^{2}$ conforming finite
element method, \cite{cakoni4,colton2,ji} mixed finite element
methods, \cite{an1} spectral-element method and \cite{gintides} the
Galerkin-type numerical method.
 Inspired by these
works, this paper further studies the $H^{2}$ conforming finite
element method and has three features as follows:\\
 \indent (1)~A complete error analysis is presented. Due to the
fact that the problem is neither elliptic nor self-adjoint, once its
error analysis  was viewed as a difficult task.
 Sun \cite{sun1}
first proved an  error estimate for the $H^{2}$ conforming finite element approximation
of (\ref{s1.5}). Following him Ji et al. \cite{ji2} presented an
accurate error estimate, but Theorems 1-2 therein are only
valid for real eigenvalues and rely on the conditions of Lemma 3.2
in \cite{sun1} which are not easy to verify in general, especially
for multiple eigenvalues. In this paper, we use the linearization
technique in \cite{tisseur} to write the weak formulation
(\ref{s1.5}) as an equivalent nonselfadjoint linear eigenvalue
problem (\ref{s2.6}). Then the $H^{2}$ conforming finite element
approximation eigenpair $(\lambda_{h}, u_{h})$ of (\ref{s2.6}) is
exactly the one of (\ref{s1.5}). Fortunately, the left-hand side
term  of (\ref{s2.6}) is a selfadjoint, continuous and coercive
sesquilinear form, thus we can use Babuska-Osborn's spectral
approximation theory \cite{babuska} to give a  complete error
analysis  for the $H^{2}$ conforming finite element approximation of
(\ref{s2.6}), namely the error analysis of (\ref{s1.5}).
 The difficulty of the error estimates in low norms lies in the
nonsymmetry of the right-hand side term of (\ref{s2.6}) that
involves derivatives.  We introduce two auxiliary problems
 and overcome this difficulty by the
Nitche technique in a subtle way. Our theoretical results are proved under general conditions
 and  are valid for arbitrary real and complex eigenvalues.
\\
\indent(2)~A finite element discretization with good algebraic
properties is presented. We give an $H^{2}$ conforming finite
element discretization (see (\ref{s3.1}) or (\ref{s5.1})) that will
lead to a  positive  definite  Hermitian and block diagonal stiff
matrix. Then we use this discretization to solve the transmission
eigenvalue problem numerically and obtain both real and complex
transmission eigenvalues of high accuracy as expected. Similar
discretizations already exist in the literatures. Colton et al.
\cite{colton2} established the $H^{2}$ conforming finite element
discretization for the formulation (\ref{s1.5}) (see (4.3) in
\cite{colton2}),
 by starting with discretizing (\ref{s1.5}) into a
quadratic algebra eigenvalue problem and then linearizing it, which
is in the opposite order of our treatment.
 Though our discrete form (\ref{s5.1}) is formally different from
the one in \cite{colton2} for the resulting mass matrix in
\cite{colton2}   is positive definite Hermitian and block diagonal,
the two discrete forms are equivalent from the linear algebra point
of view. Gintides et al. \cite{gintides} has used square root of
matrices, which is a standard tool, and a Hilbert basis in
$H_{0}^{2}(D)$ as test functions to discretizite (\ref{s1.5}) into a
linear eigenvalue problem (see (2.14) in \cite{gintides}) and proved
the convergence for the approximate eigenvalues.\\
 \indent (3) A
two grid discretization scheme is proposed. As we know, it is
difficult to solve nonselfadjoint eigenvalue problems in general.
So, Sun \cite{sun1} adopted iterative methods to work on a series of
generalized Hermitian problems and finally to solve for the real
transmission eigenvalues. Ji et al. \cite{ji2} combined the
iterative methods in \cite{sun1} with the extended finite element
method
  to establish a multigrid algorithm for solving the real transmission eigenvalues.
 This paper establishes a two grid
discretization scheme directly for the nonselfadjoint eigenvalue
problem (\ref{s2.6}) which is suitable for  arbitrary real and
complex  eigenvalues.
 The two grid discretization scheme is an
efficient approach which was first introduced by Xu \cite{xu1} for
nonsymmetric or indefinite problems, and successfully applied to
eigenvalue problems later (see, e.g.,
\cite{dai,kolman,xu3,yang1,zhou} and the references therein). With
the two grid discretization scheme, the solution of the transmission
eigenvalue problem on a fine grid $\pi_{h}$ is reduced to the
solution of the primal and dual eigenvalue problem on a much coarser
grid $\pi_{H}$ and the solutions of two linear algebraic systems
with the same positive definite Hermitian and block diagonal
coefficient matrix on the fine grid $\pi_{h}$, and the resulting
solution still maintains an asymptotically optimal accuracy. The key
to our theoretical analysis is to find a
$(u_{H}^{*},\omega_{H}^{*})$ and prove that $|B((u_{H},\omega_{H}),
(u_{H}^{*},\omega_{H}^{*}))|$ (see Step 1 of Scheme 4.1) has a
positive lower bound uniformly with respect to the mesh size $H$,
which  is also critical for further establishing multigrid method
and adaptive algorithm. We successfully apply spectral approximation
theory to solve this
problem (see Lemma 4.1 and Remark 4.1).\\
\indent The $H^{2}$ conforming finite element discretization is easy
to implement under the package of iFEM \cite{chen} with MATLAB, and
the numerical results indicate that this discretization  is
efficient for computing transmission eigenvalues.  Moreover, we can
further improve the computational efficiency
by using the two grid discretization scheme.\\
\indent Regarding the basic theory of  finite
element methods, we refer to \cite{babuska,brenner,ciarlet,oden}.\\
\indent Throughout this paper, $C$ denotes a positive constant
independent of the mesh size $h$, which may not be the same constant
in different places. For simplicity, we use the symbol $a\lesssim b$
to mean that $a\leq C b$.

\section{A linear formulation of (\ref{s1.5})}
\indent In this paper, we suppose that the index of refraction $n\in
L^{\infty}(D)$ satisfies the following assumption
\begin{eqnarray*}
1+\delta\leq \inf_{D}n(x)\leq n(x)\leq \sup_{D}n(x)<\infty,
\end{eqnarray*}
for some constant $\delta>0$, although, with obvious changes, the
theoretical analysis in this paper
also holds for $n$ strictly less than $1$.\\
\indent Inspired by the works in \cite{cakoni4,colton2,ji}, we now
apply the linearization technique in \cite{tisseur} to write the
weak
formulation (\ref{s1.5}) as an equivalent linear formulation.\\
\indent From (\ref{s1.5}) we derive that
\begin{eqnarray}\label{s2.1}
&&(\frac{1}{n-1}\Delta u, \Delta v)_{0}+k^{2}
(\frac{1}{n-1}u, \Delta v)_{0}\nonumber\\
&&~~~~~~+k^{2}(\Delta u, \frac{n}{n-1}v)_{0}+k^{4} (\frac{n}{n-1}u,v
)_{0}=0,~~~\forall v \in~H_{0}^{2}(D).
\end{eqnarray}
Introduce an auxiliary variable
\begin{eqnarray}\label{s2.2}
\omega=k^{2}u,
\end{eqnarray}
then
\begin{eqnarray}\label{s2.3}
(\omega, z)_{0}=k^{2}(u, z)_{0},~~~\forall z\in L^{2}(D).
\end{eqnarray}
Thus, combining (\ref{s2.1}) and (\ref{s2.3}), we arrive at a linear
 formulation: Find $k^{2}\in \mathbb{C}$ and nontrivial $( u, \omega) \in
H_{0}^{2}(D)\times L^{2}(D)$ such that
\begin{eqnarray}\label{s2.4}
&&(\frac{1}{n-1}\Delta u, \Delta v)_{0}=-k^{2} (\frac{1}{n-1}u,
\Delta v)_{0}\nonumber\\
&&~~~~~~-k^{2}(\Delta u, \frac{n}{n-1}v)_{0}-k^{2}
(\frac{n}{n-1}\omega, v )_{0},~~~\forall v
\in~H_{0}^{2}(D),\\\label{s2.5}
 &&(\omega,
z)_{0}=k^{2}(u,z)_{0},~~~\forall z\in L^{2}(D).
\end{eqnarray}
This is a nonselfadjoint linear eigenvalue problem.\\
\indent Suppose that $(k^{2}, u)$ is an eigenpair of (\ref{s1.5}),
then it is obvious that $(k^{2}, u, \omega)$ is an eigenpair of
(\ref{s2.4})-(\ref{s2.5}). On the other hand, suppose that $(k^{2},
u, \omega)$ satisfies (\ref{s2.4})-(\ref{s2.5}), from (\ref{s2.5})
we get $\omega=k^{2}u$, and substituting
 it into (\ref{s2.4}) we get (\ref{s2.1})(or (\ref{s1.5})). The above argument indicates that
(\ref{s2.4})-(\ref{s2.5}) and (\ref{s1.5}) are equivalent.\\
\indent Let $H^{-l}(D)$ be the ``negative space", with norm given by
\begin{eqnarray*}
\|f\|_{-l}=\sup\limits_{0\not=v\in
H_0^{l}(D)}\frac{|(f,v)_{0}|}{\|v\|_{l}},~~~l=1,2.
\end{eqnarray*}
\indent Define the Hilbert space $\mathbf{H}=H_{0}^{2}(D)\times
L^{2}(D)$ with norm
$\|(u,\omega)\|_{\mathbf{H}}=\|u\|_{2}+\|\omega\|_{0}$, and define
$\mathbf{H}^{s}=H^{s}(D)\times H^{s-2}(D)$ with norm
$\|(u,\omega)\|_{\mathbf{H}^{s}}=\|u\|_{s}+\|\omega\|_{s-2}$,
$s=0,1$, $\mathbf{H}^{2}=\mathbf{H}$.
It's obvious that $\mathbf{H}\hookrightarrow\mathbf{H}^{s}~(s=0,1)$ compactly (see\cite{berezanskii, xu3}  ). \\
\indent Denote $\lambda=k^{2}$. Let
\begin{eqnarray*}
&&A((u,\omega),(v,z))=(\frac{1}{n-1}\Delta u, \Delta v)_{0}+(\omega, z)_{0},\\
&&B((u,\omega),(v,z))=-(\frac{1}{n-1}u, \Delta v)_{0}-(\Delta u,
\frac{n}{n-1}v)_{0}- (\frac{n}{n-1}\omega, v )_{0}+(u,z)_{0},
\end{eqnarray*}
then (\ref{s2.4})-(\ref{s2.5}) can be rewritten as: Find $\lambda\in
\mathbb{C}$, $(u,\omega)\in \mathbf{H}\backslash\{\mathbf{0}\}$ such
that
\begin{eqnarray}\label{s2.6}
A((u,\omega),(v,z)) =\lambda B((u,\omega),(v,z)),~~~\forall (v,z)\in
\mathbf{H}.
\end{eqnarray}
Thus we get the following theorem.
\begin{theorem}
If $(k^{2}, u)$ is an eigenpair of (\ref{s1.5})
 and $\omega=k^{2} u$,
 then $(k^{2}, u, \omega)$ satisfies (\ref{s2.6}),
 and if $(k^{2}, u, \omega)$ satisfies (\ref{s2.6}), then  $(k^{2}, u)$ is an eigenpair of (\ref{s1.5})
 and $\omega=k^{2}u$.
 \end{theorem}
\indent By calculation we derive for any $(u,\omega),(v,z)\in
\mathbf{H}$,
\begin{eqnarray*}
&&A((u,\omega),(v,z))=\int\limits_{D}\frac{1}{n-1}\Delta u\overline{
\Delta v}+\omega
\overline{z}dx=\overline{\int\limits_{D}\frac{1}{n-1}\Delta
v\overline{
\Delta u}+z\overline{\omega}dx}=\overline{A((v,z), (u,\omega))},\\
&&|A((u,\omega),(v,z))| \lesssim \|\Delta u\|_{0}\| \Delta
v\|_{0}+\|\omega\|_{0} \|z\|_{0}\lesssim
\|(u,\omega)\|_{\mathbf{H}}\|(v,z)\|_{\mathbf{H}},\\
&&A((u,\omega),(u,\omega))=\int\limits_{D}\frac{1}{n-1}\Delta
u\overline{ \Delta u}+\omega \overline{\omega}dx\gtrsim
\|(u,\omega)\|_{\mathbf{H}}^{2},
\end{eqnarray*}
i.e.,
$A(\cdot,\cdot)$ is a selfadjoint, continuous and coercive
sesquilinear form on $\mathbf{H}\times
\mathbf{H}$.\\
\indent We use $A(\cdot,\cdot)$ and
$\|\cdot\|_{A}=A(\cdot,\cdot)^{\frac{1}{2}}$ as an inner product and
norm on $\mathbf{H}$, respectively. Then $\|\cdot\|_{A}$ is
equivalent to the norm $\|\cdot\|_{\mathbf{H}}$.
\\
\indent When $n\in L^{\infty}(D)$, $\forall (f,g)$, $(v, z)\in
\mathbf{H}$, we deduce
\begin{eqnarray}\label{s2.7}
&&|B((f,g), (v, z))|=|-(\frac{1}{n-1}f, \Delta v)_{0}-(\Delta f,
\frac{n}{n-1}v)_{0}- (\frac{n}{n-1}g, v
)_{0}+(f,z)_{0}|\nonumber\\
&&~~~\lesssim \|\frac{1}{n-1}f\|_{0}\|v\|_{2}+ \|\frac{n}{n-1}\Delta
f\|_{-1}\|v\|_{1}+
\|\frac{n}{n-1}g\|_{-1}\|v\|_{1}+ \|f\|_{0}\|z\|_{0}\nonumber\\
&&~~~\lesssim (\|f\|_{0}+\|\frac{n}{n-1}\Delta
f\|_{-1}+\|\frac{n}{n-1}g\|_{-1})\|(v,z)\|_{\mathbf{H}}.
\end{eqnarray}
\noindent When $n\in W^{1,\infty}(D)$, $\forall (f,g)\in
\mathbf{H}^{1},~\forall (v, z)\in \mathbf{H}$, we have
\begin{eqnarray}\label{s2.8}
&&|B((f,g), (v, z))|=|(\nabla(\frac{1}{n-1}f), \nabla v)_{0}+(\nabla
f, \nabla(\frac{n}{n-1}v))_{0}- (\frac{n}{n-1}g, v
)_{0}+(f,z)_{0}|\nonumber\\
&&~~~~~~\lesssim \|f\|_{1}\|v\|_{1}+ \|f\|_{1}\|v\|_{1}+
\|g\|_{-1}\|v\|_{1}+ \|f\|_{1}\|z\|_{-1}\nonumber\\
&&~~~~~~\lesssim \|(f,g)\|_{\mathbf{H}^{1}}\|(v,z)\|_{\mathbf{H}}.
\end{eqnarray}
And when $n\in W^{2,\infty}(D)$, $\forall (f,g) \in
\mathbf{H}^{0},~\forall (v, z)\in \mathbf{H}$, we have
\begin{eqnarray}\label{s2.9}
&&|B((f,g), (v, z))|=|-(\frac{1}{n-1}f, \Delta v)_{0}-( f,
\Delta(\frac{n}{n-1}v))_{0}- (\frac{n}{n-1}g, v
)_{0}+(f,z)_{0}|\nonumber\\
&&~~~~~~\lesssim \|f\|_{0}\|v\|_{2}+ \|f\|_{0}\|v\|_{2}+
\|g\|_{-2}\|v\|_{2}+ \|f\|_{0}\|z\|_{0}\nonumber\\
&&~~~~~~\lesssim \|(f,g)\|_{\mathbf{H}^{0}}\|(v,z)\|_{\mathbf{H}}.
\end{eqnarray}
 \indent We can see from (\ref{s2.7})-(\ref{s2.9}) that
for any given $(f,g)\in \mathbf{H}^{s}$ ($s=0,1,2$), $B((f,g), (v,
z))$ is a continuous linear form on $\mathbf{H}$.\\
 \indent The
source problem associated with (\ref{s2.6}) is given by: Find
$(\psi,\varphi)\in \mathbf{H}$ such that
\begin{eqnarray}\label{s2.10}
A((\psi,\varphi),(v,z)) =B((f,g),(v,z)),~~~\forall (v,z)\in
\mathbf{H}.
\end{eqnarray}
From the Lax-Milgram theorem we know that (\ref{s2.10}) admits a
unique solution. Therefore, we define the corresponding solution
operators $T: \mathbf{H}^{s}\to \mathbf{H}$ (s=0,1,2) by
\begin{eqnarray}\label{s2.11}
A(T(f,g),(v,z)) =B((f,g), (v,z)),~~~\forall (v,z)\in \mathbf{H}.
\end{eqnarray}
Then (\ref{s2.6}) has the equivalent operator form:
\begin{eqnarray}\label{s2.12}
T(u,\omega)= {\lambda^{-1}}(u,\omega).
\end{eqnarray}
Namely, if $(\lambda,u,\omega)$ is an eigenpair of (\ref{s2.6}),
then $(\lambda,u,\omega)$ is an eigenpair of (\ref{s2.12}).
Conversely, if $(\lambda,u,\omega)$ is an eigenpair of
(\ref{s2.12}), then $(\lambda,u,\omega)$ is an eigenpair of (\ref{s2.6}).\\
\begin{theorem}
Suppose $n\in L^{\infty}(D)$, then $T: \mathbf{H}\to \mathbf{H}$ is
compact; suppose $n\in W^{2-s,\infty}(D)$ $(s=0,1)$, then $T:
\mathbf{H}^{s}\to \mathbf{H}^{s}$ is compact.
\end{theorem}
\begin{proof}
When $n\in L^{\infty}(D)$, let $(v,z)=T(f,g)$ in (\ref{s2.11}), then
from (\ref{s2.7}) we have
\begin{eqnarray*}
&&\|T(f,g)\|_{\mathbf{H}}^{2}\lesssim A(T(f,g),T(f,g)) =B((f,g),
T(f,g))\\
&&~~~ \lesssim (\|f\|_{0}+\|\frac{n}{n-1}\Delta
f\|_{-1}+\|\frac{n}{n-1}g\|_{-1})\|T(f,g)\|_{\mathbf{H}},
\end{eqnarray*}
thus
\begin{eqnarray}\label{s2.13}
\|T(f,g)\|_{\mathbf{H}} \lesssim \|f\|_{0}+\|\frac{n}{n-1}\Delta
f\|_{-1}+\|\frac{n}{n-1}g\|_{-1},~~~\forall (f,g) \in \mathbf{H}.
\end{eqnarray}
Let $\mathbf{S}=S_{1}\times S_{2}$ be any bounded set in
$\mathbf{H}$, because of the compact embedding
$H_{0}^{2}(D)\hookrightarrow L^{2}(D)$ and $L^{2}(D)\hookrightarrow
H^{-1}(D)$, then $S_{1}$ is relatively compact in $L^{2}(D)$, $\{v:
v=\frac{n}{n-1}\Delta f, f\in S_{1}\}$ is relatively compact in
$H^{-1}(D)$ and $\{v: v=\frac{n}{n-1}g, g\in S_{2}\}$ is relatively
compact in $H^{-1}(D)$. And from (\ref{s2.13}) we conclude that $T:
\mathbf{H}\to \mathbf{H}$ is
compact.\\
\indent When $n\in W^{2-s,\infty}(D)~(s=0,1)$, in (\ref{s2.11}), let
$(v,z)=T(f,g)$, then from (\ref{s2.8}) and (\ref{s2.9}) we have
\begin{eqnarray*}
 \|T(f,g)\|_{\mathbf{H}}^{2}\lesssim A(T(f,g),T(f,g)) =B((f,g), T(f,g))
\lesssim \|(f,g)\|_{\mathbf{H}^{s}}\|T(f,g)\|_{\mathbf{H}},
\end{eqnarray*}
thus
\begin{eqnarray}\label{s2.14}
\|T(f,g)\|_{\mathbf{H}} \lesssim
\|(f,g)\|_{\mathbf{H}^{s}},~~~\forall (f,g) \in
\mathbf{H}^{s}~(s=0,1),
\end{eqnarray}
which implies that $T: \mathbf{H}^{s}\to \mathbf{H}$ is continuous.
Due to the compact embedding $\mathbf{H}\hookrightarrow
\mathbf{H}^{s}$, $T: \mathbf{H}^{s}\to \mathbf{H}^{s}$ is compact.
\end{proof}
\\
 \indent Consider the dual problem of (\ref{s2.6}): Find
$\lambda^{*}\in \mathbb{C}$, $(u^{*},\omega^{*})\in
\mathbf{H}\backslash\{\mathbf{0}\}$ such that
\begin{eqnarray}\label{s2.15}
A((v,z), (u^{*},\omega^{*})) =\overline{\lambda^{*}} B((v,z),
(u^{*},\omega^{*})),~~~\forall (v,z)\in \mathbf{H}.
\end{eqnarray}
Define the corresponding solution operators $T^{*}:
\mathbf{H}^{s}\to \mathbf{H}$ by
\begin{eqnarray}\label{s2.16}
A((v,z), T^{*}(f,g)) =B((v,z), (f,g)),~~~\forall (v,z)\in
\mathbf{H}.
\end{eqnarray}
Then (\ref{s2.15}) has the equivalent operator form:
\begin{eqnarray}\label{s2.17}
T^{*}(u^{*},\omega^{*})=\lambda^{*-1}(u^{*},\omega^{*}).
\end{eqnarray}

\indent It can be proved that $T^{*}$ is the adjoint operator of $T$
in the sense of inner product $A(\cdot,\cdot)$. In fact, from
(\ref{s2.11}) and (\ref{s2.16}) we have
\begin{eqnarray}\label{s2.18}
&&A(T(f,g),(v, z))=B((f,g),(v,
z))\nonumber\\
&&~~~=A((f,g),T^{*}(v, z)),~~~\forall (f,g),(v, z)\in \mathbf{H}.
\end{eqnarray}
Hence the primal and dual eigenvalues are connected via
$\lambda=\overline{\lambda^{*}}$.

\section{The $H^{2}$ conforming finite element approximation and its error
estimates}

Let $\pi_{h}$ be a shape-regular grid of $D$ with mesh size $h$ and
$S^{h}\subset H_{0}^{2}(D)$ be a piecewise polynomial space defined
on $\pi_h$; for example, $S^{h}$ is the finite element space
associated with one of the Argyris element, the Bell element or the
Bogner-Fox-Schmit element (BFS element) (see \cite{ciarlet}). Thanks
to (\ref{s2.2}), we choose $\mathbf{H}_{h}=S^{h}\times S^{h}$. Then
$\mathbf{H}_{h}\subset \mathbf{H}$ be a conforming finite element
space. For the three finite element spaces mentioned above, since
$\bigcup_{h>0} S^{h}$ are dense in both $L^{2}(D)$ and
$H_{0}^{2}(D)$, it is obvious that the following
condition (C1) holds:\\
\indent (C1)~~If $(\psi,\varphi)\in \mathbf{H}$, then as $h\to 0$,
\begin{eqnarray*}
\inf\limits_{(v,z)\in \mathbf{H_{h}}
}\|(\psi,\varphi)-(v,z)\|_{\mathbf{H}}\to 0.
\end{eqnarray*}
From the operator interpolation theory (see e.g. \cite{brenner}), the following (C2) is also valid:\\
\indent (C2)~~If $\psi\in  H_{0}^{2}(D)\cap H^{2+r}(D)$ $(r\in
(0,2])$, then
\begin{eqnarray*}
\inf\limits_{v\in S^{h}}\|\psi-v\|_{s}\lesssim
h^{2+r-s}\|\psi\|_{2+r},~~~s=0,1,2.
\end{eqnarray*}
\indent Throughout this paper, we suppose that (C1) is valid.\\
\indent The $H^{2}$ conforming finite element approximation of
(\ref{s2.6}) is given by the following: Find $\lambda_{h}\in
\mathbb{C}$, $(u_{h},\omega_{h})\in
\mathbf{H}_{h}\backslash\{\mathbf{0}\}$ such that
\begin{eqnarray}\label{s3.1}
A((u_{h},\omega_{h}),(v,z)) =\lambda_{h}
B((u_{h},\omega_{h}),(v,z)),~~~\forall (v,z)\in \mathbf{H}_{h}.
\end{eqnarray}
\indent Consider the approximate source problem: Find
$(\psi_{h},\varphi_{h})\in \mathbf{H}_{h}\backslash\{\mathbf{0}\}$ such that
\begin{eqnarray}\label{s3.2}
A((\psi_{h},\varphi_{h}),(v,z)) =B((f,g), (v,z)),~~~\forall (v,z)\in
\mathbf{H}_{h}.
\end{eqnarray}
We introduce the corresponding solution operator: $T_{h}:
\mathbf{H}^{s}\to \mathbf{H}_{h}$ (s=0,1):
\begin{eqnarray}\label{s3.3}
A(T_{h}(f,g),(v,z)) =B((f,g), (v,z)),~~~\forall (v,z)\in
\mathbf{H}_{h}.
\end{eqnarray}
Then (\ref{s3.1}) has the operator form:
\begin{eqnarray}\label{s3.4}
T_{h}(u_{h},\omega_{h})= {\lambda_{h}^{-1}}(u_{h},\omega_{h}).
\end{eqnarray}
\indent Define the projection operators $P_{h}^{1}: H_{0}^{2}(D)\to
S^{h}$ and $P_{h}^{2}: L^{2}(D)\to S^{h}$ by
\begin{eqnarray}\label{s3.5}
&&(\frac{1}{n-1}\Delta (u-P_{h}^{1}u), \Delta v)_{0}=0,~~~\forall
v\in S^{h},\\\label{s3.6} &&(\omega-P_{h}^{2}\omega,
z)_{0}=0,~~~\forall z\in S^{h}.
\end{eqnarray}
Let
$$P_{h}(u,\omega)=(P_{h}^{1}u, P_{h}^{2}\omega),~~~\forall (u,\omega)\in \mathbf{H}.$$
Then $P_{h}: \mathbf{H}\to \mathbf{H}_{h}$, and
\begin{eqnarray}\label{s3.7}
&&A((u,\omega)-P_{h}(u,\omega),(v,z))
=A((u,\omega)-(P_{h}^{1}u,P_{h}^{2}\omega),(v,z))\nonumber\\
&&~~~=(\frac{1}{n-1}\Delta (u-P_{h}^{1}u), \Delta
v)_{0}+(\omega-P_{h}^{2}\omega, z)_{0}=0,~~~\forall (v,z)\in
\mathbf{H}_{h},
\end{eqnarray}
i.e., $P_{h}: \mathbf{H}\to \mathbf{H}_{h}$ is the Ritz projection.\\
\indent Combining (\ref{s3.7}), (\ref{s2.11}) with (\ref{s3.3}), we
deduce for any $(u, \omega)\in \mathbf{H}$ that
\begin{eqnarray*}
&&A(P_{h}T(u,\omega)-T_{h}(u,\omega),(v,z))=A(P_{h}T(u,\omega)-T(u,\omega),(v,z))\\
&&~~~~~~+A(T(u,\omega)-T_{h}(u,\omega),(v,z)) =0,~~~\forall (v,z)\in
\mathbf{H}_{h},
\end{eqnarray*}
thus we get
\begin{eqnarray}\label{s3.8}
T_{h}=P_{h}T.
\end{eqnarray}

\begin{theorem}
Let $n\in L^{\infty}(D)$, then
\begin{eqnarray}\label{s3.9}
&&\|T-T_{h}\|_{\mathbf{H}}\to 0,
\end{eqnarray}
and let $n\in W^{2-s,\infty}(D)$ , then
\begin{eqnarray}\label{s3.10}
\|T-T_{h}\|_{\mathbf{H}^{s}}\to 0,~~~s=0,1.
\end{eqnarray}
\end{theorem}
 \begin{proof}
 When $n\in L^{\infty}(D)$,
 for any $ (f,g)\in \mathbf{H}$,
from (C1) we have
\begin{eqnarray*}
\|(I-P_{h})T(f,g)\|_{\mathbf{H}}=\|(T(f,g)-P_{h}T(f,g)\|_{\mathbf{H}}\lesssim\inf\limits_{(v,z)\in
\mathbf{H}_{h} }\|(T(f,g)-(v,z)\|_{\mathbf{H}}\to 0.
\end{eqnarray*}
Since $T: \mathbf{H}\to \mathbf{H}$ is compact, $T\{(f,g)\in
\mathbf{H}, \|(f,g)\|_{\mathbf{H}=1}\}$ is relatively compact.
 Thus by the definition of
operator norm we have
\begin{eqnarray*}
&&\|T-T_{h}\|_{\mathbf{H}}
=\sup\limits_{(f,g)\in \mathbf{H}, \|(f,g)\|_{\mathbf{H}}=1}\|(T-T_{h})(f,g)\|_{\mathbf{H}}\nonumber\\
&&~~~=\sup\limits_{(f,g)\in \mathbf{H},
\|(f,g)\|_{\mathbf{H}}=1}\|(I-P_{h})T(f,g)\|_{\mathbf{H}}\to 0.
\end{eqnarray*}
When $n\in W^{2-s,\infty}(D)$, for $s=0,1$, we have
\begin{eqnarray*}
\|(I-P_{h})T(f,g)\|_{\mathbf{H^{s}}}\lesssim\|(I-P_{h})T(f,g)\|_{\mathbf{H}}\to
0.
\end{eqnarray*}
Since $T: \mathbf{H}^{s}\to \mathbf{H}^{s}$ is compact, $T\{(f,g)\in
\mathbf{H}^{s}, \|(f,g)\|_{\mathbf{H}^{s}=1}\}$ is relatively
compact, thus we have
\begin{eqnarray*}
&&\|T-T_{h}\|_{\mathbf{H}^{s}}
=\sup\limits_{(f,g)\in \mathbf{H}^{s}, \|(f,g)\|_{\mathbf{H}^{s}}=1}\|(T-T_{h})(f,g)\|_{\mathbf{H}^{s}}\nonumber\\
&&~~~=\sup\limits_{(f,g)\in \mathbf{H}^{s},
\|(f,g)\|_{\mathbf{H}^{s}}=1}\|(I-P_{h})T(f,g)\|_{\mathbf{H}^{s}}\to
0.
\end{eqnarray*}
The proof is completed.
\end{proof}

\indent The conforming finite element approximation of (\ref{s2.15})
is given by: Find $\lambda_{h}^{*} \in \mathbb{C}$,
$(u_{h}^{*},\omega_{h}^{*})\in
\mathbf{H}_{h}\backslash\{\mathbf{0}\}$ such that
\begin{eqnarray}\label{s3.11}
A((v,z), (u_{h}^{*},\omega_{h}^{*})) =\overline{\lambda_{h}^{*}}
B((v,z), (u_{h}^{*},\omega_{h}^{*})),~~~\forall (v,z)\in
\mathbf{H}_{h}.
\end{eqnarray}

Define the solution operator $T_{h}^{*}: \mathbf{H}^{s}\to
\mathbf{H}_{h}$ satisfying
\begin{eqnarray}\label{s3.12}
A((v,z), T_{h}^{*}(f,g))=B((v,z), (f,g)),~~~\forall~(v,z)\in
\mathbf{H}_{h}.
\end{eqnarray}
 (\ref{s3.11}) has the following equivalent operator form
\begin{eqnarray}\label{s3.13}
T_{h}^{*}(u_{h}^{*},\omega_{h}^{*})=\lambda_{h}^{*-1}(u_{h}^{*},\omega_{h}^{*}).
\end{eqnarray}
\indent It can be proved that $T_{h}^{*}$ is the adjoint operator of
$T_{h}$ in the sense of inner product $A(\cdot,\cdot)$. In fact,
from (\ref{s3.3}) and (\ref{s3.12}) we have
\begin{eqnarray}\label{s3.14}
&&A(T_{h}(f,g),(v, z))=B((f,g),(v,
z))\nonumber\\
&&~~~=A((f,g),T_{h}^{*}(v, z)),~~~\forall (f,g),(v, z)\in
\mathbf{H}_{h}.
\end{eqnarray}
Hence, the primal and dual eigenvalues are connected via
$\lambda_{h}=\overline{\lambda_{h}^{*}}$.\\
 \indent We need the following
lemma (see Lemma 5 on page 1091 of \cite{dunford}).
\begin{lemma}
Let $\|T_{h}-T\|_{\mathbf{H}}\to 0$. Let $\{\lambda_{j}\}$  be  an
enumeration of the eigenvalues of $T$, each repeated according to
its multiplicity. Then there exist enumerations $\{\lambda_{j,h}\}$
of the eigenvalues of $T_{h}$, with repetitions according to
multiplicity, such that $\lambda_{j,h}\to \lambda_{j}$~($j\geq 1$).
 \end{lemma}

 \indent In this paper we suppose that
$\{\lambda_{j}\}$ and $\{\lambda_{j,h}\}$  satisfy the above
lemma, and let $\lambda=\lambda_{k}$ be the kth eigenvalue with the
algebraic multiplicity $q$ and the ascent $\alpha$,
$\lambda_{k}=\lambda_{k+1}=\cdots,\lambda_{k+q-1}$. Since
$\|T_{h}-T\|_{\mathbf{H}}\to 0$, $q$ eigenvalues
$\lambda_{k,h},\cdots,\lambda_{k+q-1,h}$ of (\ref{s3.1}) will
converge to $\lambda$.\\
 \indent Let $E$ be the
spectral projection associated with $T$ and $\lambda$, then
$R(E)=N((\lambda^{-1}-T)^{\alpha})$ is the space of generalized
eigenvectors associated with $\lambda$ and $T$, where $R$ denotes
the range and $N$ denotes the null space. Let $E_{h}$ be the
spectral projection associated with $T_{h}$ and the eigenvalues
$\lambda_{k,h},\cdots,\lambda_{k+q-1,h}$; let $M_{h}(\lambda_{i,h})$
be the space of generalized eigenvectors associated with
$\lambda_{i,h}$ and $T_{h}$, and let
$M_{h}(\lambda)=\sum_{i=k}^{k+q-1}M_{h}(\lambda_{i,h})$,  then
$R(E_{h})=M_{h}(\lambda)$ if $h$ is small enough.
 In view of the
dual problem (\ref{s2.15}) and (\ref{s3.11}), the definitions of
$E^{*}$, $R(E^{*})$, $E_{h}^{*}$, $M_{h}(\lambda_{i,h}^{*})$,
$M_{h}(\lambda^{*})$ and $R(E_{h}^{*})$ are analogous to $E$,
$R(E)$, $E_{h}$, $M_{h}(\lambda_{i,h})$, $M_{h}(\lambda)$ and
$R(E_{h})$ (see
\cite{babuska}).\\
\indent Given two closed subspaces $R$ and $U$, denote
\begin{eqnarray*}
&&\delta(R,U)=\sup\limits_{(u,\omega)\in R\atop\|(u,\omega)\|_{\mathbf{H}}=1}\inf\limits_{(v,z)\in U}\|(u,\omega)-(v,z)\|_{\mathbf{H}},\\
&&\theta(R,U)_{s}=\sup\limits_{(u,\omega)\in
R\atop\|(u,\omega)\|_{\mathbf{H^{s}}}=1}\inf\limits_{(v,z)\in
U}\|(u,\omega)-(v,z)\|_{\mathbf{H^{s}}},~~~s=0,1.
\end{eqnarray*}
We define the gaps between $R(E)$ and $R(E_{h})$ in
$\|\cdot\|_{\mathbf{H}}$ as
\begin{eqnarray*}
 \widehat{\delta}(R(E),R(E_{h}))=\max\{\delta(R(E),R(E_{h})),\delta(R(E_{h}),R(E))\},
\end{eqnarray*}
and in $\|\cdot\|_{\mathbf{H}^{s}}$ as
\begin{eqnarray*}
 \widehat{\theta}(R(E),R(E_{h}))_{s}=\max\{\theta(R(E),R(E_{h}))_{s},\theta(R(E_{h}),R(E))_{s}\}.
\end{eqnarray*}
Define
\begin{eqnarray*}
&&\varepsilon_{h}(\lambda)=\sup\limits_{(u,\omega)\in
R(E)\atop\|(u,\omega)\|_{\mathbf{H}}=1} \inf\limits_{(v,z)\in
\mathbf{H}_{h}}\|(u,\omega)-(v,z)\|_{\mathbf{H}},\\
&&\varepsilon_{h}^{*}(\lambda^{*})=\sup\limits_{(u^{*},\omega^{*})\in
R(E^{*})\atop\|(u^{*},\omega^{*})\|_{\mathbf{H}}=1}
\inf\limits_{(v,z)\in
\mathbf{H}_{h}}\|(u^{*},\omega^{*})-(v,z)\|_{\mathbf{H}}.
\end{eqnarray*}
It follows directly from (C1) that
\begin{eqnarray*}
\varepsilon_{h}(\lambda)\to 0~(h\to
0),~~~\varepsilon_{h}^{*}(\lambda^{*})\to 0~(h\to 0).
\end{eqnarray*}
\noindent Suppose that $R(E), R(E^{*})\subset \mathbf{H}\cap
(H^{2+r}(D))^{2}$ $(r\in (0,2])$, then from (C2) we get
\begin{eqnarray}\label{s3.15}
 \varepsilon_{h}(\lambda)\lesssim h^{r},~~~\varepsilon_{h}^{*}(\lambda^{*})\lesssim h^{r}.
\end{eqnarray}
Further suppose that $R(E),  R(E^{*})\subset (H^{6}(D))^{2}$,
$S^{h}\subset H_{0}^{2}(D)$ is the Argyris finite element space,
then from the interpolation theory we have
\begin{eqnarray}\label{s3.16}
 \varepsilon_{h}(\lambda)\lesssim h^{4},~~~\varepsilon_{h}^{*}(\lambda^{*})\lesssim h^{4}.
\end{eqnarray}
Note that when the functions in $R(E)$ and $R(E^{*})$ are piecewise
smooth (\ref{s3.15}) and (\ref{s3.16}) are also valid.\\
\indent Thanks to \cite{babuska},  we get the following Theorem 3.3.
\begin{theorem}
 Suppose $n\in L^{\infty}(D)$, then
\begin{eqnarray}\label{s3.17}
&&\widehat{\delta}(R(E),
R(E_{h}))\lesssim\varepsilon_{h}(\lambda),\\\label{s3.18}
&&|\lambda-(\frac{1}{q}\sum\limits_{j=k}^{k+q-1}\lambda_{j,h}^{-1})^{-1}|\lesssim
\varepsilon_{h}(\lambda)\varepsilon_{h}^{*}(\lambda^{*}),\\\label{s3.19}
&&|\lambda-\lambda_{j,h}|\lesssim
[\varepsilon_{h}(\lambda)\varepsilon_{h}^{*}(\lambda^{*})]^{\frac{1}{\alpha}},~~~j=k,k+1,\cdots,k+q-1.
\end{eqnarray}
Suppose $(u_{h},\omega_{h})$ with $\|(u_{h},\omega_{h})\|_{A}=1$
is an eigenfunction corresponding to $\lambda_{j,h}$
($j=k,k+1,\cdots,k+q-1$), then there exists an eigenfunction
$(u,\omega)$ corresponding to $\lambda$, such that
\begin{eqnarray}\label{s3.20}
&&\|(u_{h},\omega_{h})-(u,\omega)\|_{\mathbf{H}}\lesssim
\varepsilon_{h}(\lambda)^{\frac{1}{\alpha}}.
\end{eqnarray}
\end{theorem}
\begin{proof}
From Theorem 3.1 we know $\|T-T_{h}\|_{\mathbf{H}}\to 0$ $(h\to
 0)$, thus from Theorem 8.1, Theorem 8.2, Theorem 8.3 and Theorem 8.4
of \cite{babuska} we get the desired results (\ref{s3.17}),
(\ref{s3.18}), (\ref{s3.19}) and (\ref{s3.20}), respectively.
\end{proof}

\indent Next we discuss the error estimates in the $\|\cdot\|_{\mathbf{H}^{s}}$ $(s=0,1)$ norm by using the Aubin-Nitsche technique.\\
\indent We need the following regularity assumption:\\
\indent {\bf R(D)}.~~ For any $\xi\in H^{-s}(D)$ ($s=0,1$), there
exists $\psi\in H^{2+r_{s}}(D)$ satisfying
\begin{eqnarray*}
\Delta(\frac{1}{n-1}\Delta \psi)=\xi,~in~D;~~~\psi=\frac{\partial
\psi}{\partial \nu}=0~ on~\partial D,
\end{eqnarray*}
and
\begin{eqnarray}\label{s3.21}
\|\psi\|_{2+r_{s}}\leq C_{p} \|\xi\|_{-s},~~~s=0,1
\end{eqnarray}
where $r_{1}\in (0,1]$, $r_{0}\in (0,2]$, $C_{p}$ denotes the prior
constant dependent on $n(x)$ and $D$ but
independent of the right-hand side $\xi$ of the equation.\\
 \indent It is easy to see that (\ref{s3.21}) is valid with $r_{s}=2-s$
  when $n$ and $\partial D$ are appropriately smooth.
When $n$ is a constant and $D\subset \mathbb{R}^{2}$   is a convex
polygon, from Theorem 2 in \cite{blum} we can get $r_{1}=1$ and if
the inner angle at each critical boundary point is smaller than
$126.283696...^0$ then $r_{0}=2$.\\
\indent Consider the auxiliary boundary value problem: Find $ \phi_{1}\in H^2_0(D)$ such that
\begin{eqnarray}\label{s3.22}
&&\Delta(\frac{1}{n-1}\Delta \phi_{1})=-\Delta
(u-P_{h}^{1}u),~~~in~D,\\\label{s3.23}
 &&\phi_{1}=\frac{\partial
\phi_{1}}{\partial \nu}=0,~~~ on~\partial D.
\end{eqnarray}
Let R(D) hold, then
\begin{eqnarray*}
\|\phi_{1}\|_{2+r_{1}}\lesssim \|\Delta(u-P_{h}^{1}u)\|_{-1}\lesssim
\|u-P_{h}^{1}u\|_{1}.
\end{eqnarray*}
\indent Consider the auxiliary boundary value problem:
\begin{eqnarray}\label{s3.24}
&&\Delta(\frac{1}{n-1}\Delta
\phi_{2})=u-P_{h}^{1}u,~~~in~D,\\\label{s3.25}
 &&\phi_{2}=\frac{\partial
\phi_{2}}{\partial \nu}=0,~~~ on~\partial D.
\end{eqnarray}
Let R(D) hold, then
\begin{eqnarray*}
\|\phi_{2}\|_{2+r_{0}}\lesssim \|u-P_{h}^{1}u\|_{0}.
\end{eqnarray*}

\begin{lemma}
Suppose that   R(D) and (C2) are valid $(s=0,1)$, then
 for $(u, \omega)\in \mathbf{H}$,
\begin{eqnarray}\label{s3.26}
\|(u, \omega)-P_{h}(u,\omega)\|_{\mathbf{H}^{s}} \lesssim
h^{r_{s}}\| (u,\omega)-P_{h}(u,\omega)\|_{\mathbf{H}},~~~s=0,1.
\end{eqnarray}
\end{lemma}
\begin{proof}
The weak form of (\ref{s3.22})-(\ref{s3.23}) is
\begin{eqnarray*}
&&(\Delta v, \frac{1}{n-1}\Delta \phi_{1})_{0}=(\nabla v, \nabla
(u-P_{h}^{1}u))_{0},~~~\forall v\in H_{0}^{2}(D).
\end{eqnarray*}
Let $v=u-P_{h}^{1}u$, then
\begin{eqnarray}\label{s3.27}
(\Delta (u-P_{h}^{1}u), \frac{1}{n-1}\Delta
(\phi_{1}-P_{h}^{1}\phi_{1}))_{0}=(\nabla (u-P_{h}^{1}u), \nabla
 (u-P_{h}^{1}u))_{0},
\end{eqnarray}
which leads to
\begin{eqnarray*}
&&\|\nabla
(u-P_{h}^{1}u)\|_{0}^{2}\lesssim\|\Delta(u-P_{h}^{1}u)\|_{0}
\|\Delta
(\phi_{1}-P_{h}^{1}\phi_{1})\|_{0}\\
&&~~~\lesssim\|\Delta(u-P_{h}^{1}u)\|_{0} h^{r_{1}}\|
u-P_{h}^{1}u\|_{1},
\end{eqnarray*}
thus
\begin{eqnarray}\label{s3.28}
\|u-P_{h}^{1}u\|_{1} \lesssim h^{r_{1}}\| u-P_{h}^{1}u\|_{2}.
\end{eqnarray}
The weak form of (\ref{s3.24})-(\ref{s3.25}) is
\begin{eqnarray*}
&&(\Delta v, \frac{1}{n-1}\Delta \phi_{2})_{0}=( v,
u-P_{h}^{1}u)_{0},~~~\forall v\in H_{0}^{2}(D).
\end{eqnarray*}
Let $v=u-P_{h}^{1}u$, then
\begin{eqnarray}\label{s3.29}
(\Delta (u-P_{h}^{1}u), \frac{1}{n-1}\Delta
(\phi_{2}-P_{h}^{1}\phi_{2}))_{0}=( u-P_{h}^{1}u, u-P_{h}^{1}u)_{0},
\end{eqnarray}
thus
\begin{eqnarray*}
\|u-P_{h}^{1}u\|_{0}^{2}\lesssim\|\Delta (u-P_{h}^{1}u)\|_{0}
\|\Delta (\phi_{2}-P_{h}^{1}\phi_{2})\|_{0} \lesssim
\|u-P_{h}^{1}u\|_{2} h^{r_{0}}\|u-P_{h}^{1}u\|_{0}.
\end{eqnarray*}
and
\begin{eqnarray}\label{s3.30}
&&\|u-P_{h}^{1}u\|_{0}\lesssim h^{r_{0}}\|u-P_{h}^{1}u\|_{2}.
\end{eqnarray}
Combining (\ref{s3.28}) and (\ref{s3.30}) we get
\begin{eqnarray}\label{s3.31}
\|u-P_{h}^{1}u\|_{s} \lesssim h^{r_{s}}\|
u-P_{h}^{1}u\|_{2},~~~s=0,1.
\end{eqnarray}
By the definition of the negative norm,
\begin{eqnarray}\nonumber
&& \|\omega-P_{h}^{2}\omega\|_{s-2}=\sup\limits_{\xi\in
H_{0}^{2-s}(D)}\frac{(\omega-P_{h}^{2}\omega,
\xi)_{0}}{\|\xi\|_{2-s}}\\\label{s3.32}
 &&~~~\lesssim\sup\limits_{\xi\in
H_{0}^{2-s}(D)}\frac{\|\omega-P_{h}^{2}\omega\|_{0}
\|\xi-P_{h}^{2}\xi\|_{0}}{\|\xi\|_{2-s}}\lesssim
h^{2-s}\|\omega-P_{h}^{2}\omega\|_{0},~~~s=0,1.
\end{eqnarray}
Combining (\ref{s3.31}) with (\ref{s3.32}) we deduce that
\begin{eqnarray*}
&&\|(u, \omega)-P_{h}(u,\omega)\|_{\mathbf{H}^{s}} =\|(u
-P_{h}^{1}u, \omega-P_{h}^{2}\omega)\|_{\mathbf{H}^{s}}\nonumber\\
&&~~~\lesssim \|u -P_{h}^{1}u\|_{s}+
\|\omega-P_{h}^{2}\omega\|_{s-2} \lesssim h^{r_{s}}\|
u-P_{h}^{1}u\|_{2}+h^{2-s}\|\omega-P_{h}^{2}\omega\|_{0}\nonumber\\
&&~~~\lesssim (h^{r_{s}}+h^{2-s})\|
(u,\omega)-P_{h}(u,\omega)\|_{\mathbf{H}},~~s=0,1,
\end{eqnarray*}
noting that $h^{2-s}\lesssim h^{r_{s}}$, we obtain (\ref{s3.26}).
\end{proof}

\begin{theorem}
Suppose that R(D) and $(C2)$ are valid, and $n \in
W^{2-s,\infty}(D)$ ($s=0,1$). Then
\begin{eqnarray}\label{s3.33}
\widehat{\theta}(R(E), R(E_{h}))_{s}
 \lesssim h^{r_{s}}\varepsilon_{h}(\lambda),~~~s=0,1.
\end{eqnarray}
Let $(u_{h},\omega_{h})$ with $\|(u_{h},\omega_{h})\|_{A}=1$ be an
eigenfunction corresponding to $\lambda_{h}$, then there exists
eigenfunction $(u,\omega)$ corresponding to $\lambda$, such that
\begin{eqnarray}\label{s3.34}
\|(u_{h},\omega_{h})-(u,\omega)\|_{\mathbf{H}^{s}}\lesssim
(h^{r_{s}}\varepsilon_{h}(\lambda))^{\frac{1}{\alpha}},~~~s=0,1;
\end{eqnarray}
further let $\alpha=1$, then
\begin{eqnarray}\label{s3.35}
\|(u_{h},\omega_{h})-(u,\omega)\|_{\mathbf{H}^{s}}\lesssim h^{r_{s}}
\|(u_{h},\omega_{h})-(u,\omega)\|_{\mathbf{H}},~~~s=0,1.
\end{eqnarray}
\end{theorem}
\begin{proof}
Note that in $R(E)$ the norm $\|\cdot\|_{\mathbf{H}^{s}}$ is
equivalent to the norm $\|\cdot\|_{\mathbf{H}}$, and $TR(E)\subset
R(E)$, by (\ref{s3.26}) we deduce that
\begin{eqnarray}\label{s3.36}
&& \|(T-T_{h})|_{R(E)}\|_{\mathbf{H}^{s}}=
\sup\limits_{(u,\omega)\in R(E)}
\frac{\|T(u,\omega)-T_{h}(u,\omega)\|_{\mathbf{H}^{s}}}{\|(u,\omega)\|_{\mathbf{H}^{s}}}\nonumber\\
&&~~~\lesssim \sup\limits_{(u,\omega)\in R(E)}
\frac{\|T(u,\omega)-P_{h}T(u,\omega)\|_{\mathbf{H}^{s}}}{\|(u,\omega)\|_{\mathbf{H}}}\nonumber\\
&&~~~\lesssim h^{r_{s}}\sup\limits_{(u,\omega)\in R(E)}
\frac{\|T(u,\omega)-P_{h}T(u,\omega)\|_{\mathbf{H}}}{\|(u,\omega)\|_{\mathbf{H}}}\nonumber\\
&&~~~=h^{r_{s}}\sup\limits_{(u,\omega)\in R(E)}
\frac{\|T(u,\omega)-P_{h}T(u,\omega)\|_{\mathbf{H}}}{\|T(u,\omega)\|_{\mathbf{H}}}
\frac{\|T(u,\omega)\|_{\mathbf{H}}}{\|(u,\omega)\|_{\mathbf{H}}}\lesssim
h^{r_{s}}\varepsilon_{h}(\lambda).
\end{eqnarray}
Thanks to Theorem 3.1, Theorem 7.1 in \cite{babuska} and
(\ref{s3.36}) we deduce
\begin{eqnarray*}
\widehat{\theta}(R(E), R(E_{h}))_{s}\lesssim
\|(T-T_{h})|_{R(E)}\|_{\mathbf{H}^{s}}\lesssim
h^{r_{s}}\varepsilon_{h}(\lambda).
\end{eqnarray*}
And from Theorem 7.4 in \cite{babuska} we get
\begin{eqnarray*}
\|(u_{h},\omega_{h})-(u,\omega)\|_{\mathbf{H}^{s}}\lesssim
\|(T-T_{h})|_{R(E)}\|_{\mathbf{H}^{s}}^{\frac{1}{\alpha}}\lesssim
(h^{r_{s}}\varepsilon_{h}(\lambda))^{\frac{1}{\alpha}}.
\end{eqnarray*}
When $\alpha=1$, we have
\begin{eqnarray*}
\|(u_{h},\omega_{h})-(u,\omega)\|_{\mathbf{H}^{s}}\lesssim
\|(T-T_{h})(u,\omega)\|_{\mathbf{H}^{s}}\lesssim
\|(u,\omega)-P_{h}(u,\omega)\|_{\mathbf{H}^{s}},
\end{eqnarray*}
thus from (\ref{s3.26}) and
$\|(u,\omega)-P_{h}(u,\omega)\|_{\mathbf{H}}\lesssim
\|(u,\omega)-(u_{h},\omega_{h})\|_{\mathbf{H}}$ we obtain
(\ref{s3.35}). This completes the proof.
\end{proof}

\indent {\bf Remark 3.1.}~~Using the same argument as in this
section
 we can prove the error estimates of finite element approximation for the dual problem (\ref{s2.15}),
for example, there hold the following two estimates
\begin{eqnarray}\label{s3.37}
&&\|(u_{h}^{*},\omega_{h}^{*})-(u^{*},\omega^{*})\|_{\mathbf{H}}\lesssim
\varepsilon_{h}^{*}(\lambda^{*})^{\frac{1}{\alpha}},\\\label{s3.38}
&&\|(u_{h}^{*},\omega_{h}^{*})-(u^{*},\omega^{*})\|_{\mathbf{H}^{s}}\lesssim
(h^{r_{s}}\varepsilon_{h}^{*}(\lambda^{*}))^{\frac{1}{\alpha}},~~~s=0,1.
\end{eqnarray}

\section{Two grid discretization scheme}

\indent In this section we use the two grid discretization scheme to treat
transmission eigenvalues problem.\\
 \indent{\bf Definition 4.1.}~~ $\forall$ $(v,z),(v^{*},z^{*})\in \mathbf{H}
$, $B((v,z),(v^{*},z^{*}))\neq 0$, define
$
\frac{A((v,z),(v^{*},z^{*}))}{B((v,z),(v^{*},z^{*}))}
$
as the generalized Rayleigh quotient of $(v,z)$ and $(v^{*},z^{*})$.\\
\indent We now outline the two grid discretization scheme.\\
\indent{\bf Scheme 4.1.}~~Two grid discretization scheme.\\
\indent {\bf Step 1.} Solve (\ref{s3.1}) on a coarse grid $\pi_{H}$:
Find $\lambda_{H}\in \mathbb{C}, (u_{H},\omega_{H})\in
\mathbf{H}_{H}$ such that $\|(u_{H},\omega_{H})\|_{{A}}=1$ and
\begin{eqnarray}\label{s4.1y}
A((u_{H},\omega_{H}),(v,z)) =\lambda_{H}
B((u_{H},\omega_{H}),(v,z)),~~~\forall (v,z)\in \mathbf{H}_{H},
\end{eqnarray}
and find $(u_{H}^{*}, \omega_{H}^{*})\in
M_{H}(\lambda^*)$ according to Lemma 4.1 and Remark 4.1.\\
\indent {\bf Step 2.} Solve two linear boundary value problems on a
fine grid $\pi_{h}$: Find $(u^h,\omega^{h})\in \mathbf{H}_{h}$ such
that
\begin{eqnarray*}
A((u^h,\omega^{h}),
(v,z))=\lambda_{H}B((u_{H},\omega_{H}),(v,z)),~~~\forall (v,z)\in
\mathbf{H}_{h},
\end{eqnarray*}
and find $(u^{h*},\omega^{h*})\in \mathbf{H}_{h}$ such that
\begin{eqnarray*}
A((v,z), (u^{h*},\omega^{h*}))=\lambda_{H}
B((v,z),(u_{H}^{*},\omega_{H}^{*})),~~~\forall (v,z)\in
\mathbf{H}_{h}.
\end{eqnarray*}
\indent {\bf Step 3.} Compute the generalized Rayleigh quotient
$
 \lambda^{h}=\frac{A((u^h,\omega^{h}),(u^{h*},\omega^{h*}))}{ B((u^h,\omega^{h}),(u^{h*},\omega^{h*}))}.
$
\\\\
\indent A basic condition to the two grid discretization scheme is
that $|B((u_{H},\omega_{H}), (u_{H}^{*},\omega_{H}^{*}))|$ has a
positive lower bound uniformly with respect to $H$ (see, e.g.,
Theorem 3.5 in \cite{kolman}). The following Lemma 4.1 guarantees
this condition. Using this condition, in the following Theorem 4.3
we can also prove $|B((u^h,\omega^{h}),(u^{h*},\omega^{h*}))|$ has a
positive lower bound uniformly with respect to $h$.

\begin{lemma}
Suppose that $\lambda_{H}=\lambda_{j,H}$ $(j=k,k+1,\cdots,k+q-1)$,
and $(u_{H},\omega_{H})$ is an eigenfunction corresponding to
$\lambda_{H}$.
 Let
$(u_{H}^{-},\omega_{H}^{-})$ be the orthogonal projection of
$(u_{H},\omega_{H})$ to $M_{H}(\lambda^*)$ in the sense of inner
product $A(\cdot,\cdot)$, and let
\begin{eqnarray}\label{s4.1}
(u_{H}^{*},\omega_{H}^{*})=\frac{(u_{H}^{-},\omega_{H}^{-})}{\|(u_{H}^{-},\omega_{H}^{-})\|_{A}}.
\end{eqnarray}
Then when $H$ is small enough $|B((u_{H},\omega_{H}),
(u_{H}^{*},\omega_{H}^{*}))|$ has a positive lower bound uniformly
with respect to $H$.
\end{lemma}
\begin{proof}
 Define $f((v,z))=A(E(v,z), (u',\omega'))$, where
$(u',\omega')=\frac{E(u_{H},\omega_{H})}{\|E(u_{H},\omega_{H})\|_{A}}$.
Since for any $(v,z)\in \mathbf{H}$
\begin{eqnarray*}
|f((v,z))|\leq \|E(v,z)\|_{A}\leq \|E\|_{A}\|(v,z)\|_{A},
\end{eqnarray*}
$f$ is a linear and bounded functional on $\mathbf{H}$ and
$\|f\|_{A}\leq \|E\|_{A}$. Using Riesz Theorem, we know  there
exists $(\widetilde{u},\widetilde{\omega})\in \mathbf{H}$ satisfying
\begin{eqnarray*}
A((v,z), (\widetilde{u}, \widetilde{\omega}))=A(E(v,z), (u',
\omega')),~~~\|(\widetilde{u}, \widetilde{\omega})\|_{A}\leq  \|E\|_A .
\end{eqnarray*}
Now, for any $ (v,z)\in \mathbf{H}$,
\begin{eqnarray*}
&&A((v,z),(\lambda^{*-1}-T^{*})^{\alpha}(\widetilde{u}, \widetilde{\omega})
)=A((\lambda^{-1}-T)^{\alpha}(v,z),(\widetilde{u}, \widetilde{\omega}) )\\
&&~~~=A(E(\lambda^{-1}-T)^{\alpha}(v,z),(u', \omega')
)=A((\lambda^{-1}-T)^{\alpha}E(v,z),(u', \omega') )=0,
\end{eqnarray*}
i.e., $(\lambda^{*-1}-T^{*})^{\alpha}(\widetilde{u},
\widetilde{\omega})=0$. Hence $(\widetilde{u},
\widetilde{\omega})\in R(E^{*})$. Let
$(u_{H}',\omega_{H}')=\frac{E_{H}^{*}(\widetilde{u},
\widetilde{\omega})}{\overline{A((u_{H},\omega_{H}),E_{H}^{*}(\widetilde{u},
\widetilde{\omega}))}}$, then $(u_{H}',\omega_{H}')\in
R(E_{H}^{*})=M_{H}(\lambda^*)$,
\begin{eqnarray}\label{s4.2}
1&=&A((u_{H},\omega_{H}),(u_{H}',\omega_{H}'))\nonumber\\
&=&A((u_{H},\omega_{H})-(u_{H}^{-},\omega_{H}^{-})+(u_{H}^{-},\omega_{H}^{-}),(u_{H}',\omega_{H}'))\nonumber\\
&=&A((u_{H}^{-},\omega_{H}^{-}),(u_{H}',\omega_{H}'))\leq
\|(u_{H}^{-},\omega_{H}^{-})\|_{A}\|(u_{H}',\omega_{H}')\|_{A}.
\end{eqnarray}
From \cite{babuska,chatelin}, if $H\to 0$ we have
\begin{eqnarray*}
&&\|(E_{H}^{*}-E^{*})|_{R(E^{*})}\|_{A}\to 0,\nonumber\\
&&\|E(u_{H},\omega_{H})\|_{A}
=\|E(u_{H},\omega_{H})-(u_{H},\omega_{H})+(u_{H},\omega_{H})\|_{A}
\to 1,
\end{eqnarray*}
and
\begin{eqnarray*}
&&A((u_{H},\omega_{H}),E_{H}^{*}(\widetilde{u},\widetilde{\omega}))
=A((u_{H},\omega_{H}),(E_{H}^{*}-E^{*})(\widetilde{u},\widetilde{\omega}))+A((u_{H},\omega_{H}),E^{*}(\widetilde{u},\widetilde{\omega}))\nonumber\\
&&~~~=A((u_{H},\omega_{H}),(E_{H}^{*}-E^{*})(\widetilde{u},\widetilde{\omega}))+\|E(u_{H},\omega_{H})\|_{A}
\to 1.
\end{eqnarray*}
Thus, there is a constant $C_{0}$ independent of $H$ such that
\begin{eqnarray}\label{s4.3}
\|(u_{H}',\omega_{H}')\|_{A}=\frac{\|E_{H}^{*}(\widetilde{u},
\widetilde{\omega})\|_{A}}{|A((u_{H},\omega_{H}),E_{H}^{*}(\widetilde{u}, \widetilde{\omega}))|}\leq
C_{0}.
\end{eqnarray}
From (\ref{s3.1}), (\ref{s4.2}) and (\ref{s4.3}), deduce
\begin{eqnarray*}
&&|B((u_{H},\omega_{H}),
(u_{H}^{*},\omega_{H}^{*}))|=|\lambda_{H}^{-1}|
|A((u_{H},\omega_{H}),
(u_{H}^{*},\omega_{H}^{*}))|\nonumber\\
&&~~~=|\lambda_{H}^{-1}||A((u_{H},\omega_{H}),\frac{(u_{H}^{-},\omega_{H}^{-})}{\|(u_{H}^{-},\omega_{H}^{-})\|_{A}})|
=|\lambda_{H}^{-1}||A((u_{H}^{-},\omega_{H}^{-}),\frac{(u_{H}^{-},\omega_{H}^{-})}{\|(u_{H}^{-},\omega_{H}^{-})\|_{A}})|
\\
&&~~~=|\lambda_{H}^{-1}|\|(u_{H}^{-},\omega_{H}^{-})\|_{A} \geq
\frac{1}{ |\lambda_{H}|
 \|(u_{H}',\omega_{H}')\|_{A} } \geq
\frac{1}{C_{0}|\lambda_{H}|}.
\end{eqnarray*}
Then when $H$ is small enough $|B((u_{H},\omega_{H}),
(u_{H}^{*},\omega_{H}^{*}))|$ has a positive lower bound uniformly
with respect to $H$.
\end{proof}

\indent The above lemma is fundamental for studying two grid method
as well as multigrid method and adaptive algorithm for the
transmission eigenvalue problem and  many $2m$th order
nonselfadjoint elliptic eigenvalue problems.

 \indent{\bf Remark 4.1.}~~ Computational method for
 $(u_{H}^{*},\omega_{H}^{*})$.\\
 \indent {\bf Step 1.}~Find a basis
$\{(\psi_{i},\varphi_{i})\}_{i=k}^{k+q-1}$ in $M_{H}(\lambda^*)$:\\
\indent How to seek this basis efficiently is an important issue of
linear algebra. When the ascent of $\lambda_{i,H}$ is equal to 1
$(i=k,\cdots,k+q-1)$, we can use the Arnoldi algorithm
\cite{golub,lehoucq,saad} to solve the dual problem of (\ref{s4.1y})
and obtain this basis and meanwhile  MATLAB has provided   implemented Arnoldi
solvers ``sptarn" and ``eigs"; we can also use the two sided Arnoldi
algorithm in \cite{cullum} to compute both left and right
eigenvectors of (\ref{s4.1y}) at the same time, and obtain a basis
$\{(u_{i,H},\omega_{i,H})\}_{i=k}^{k+q-1}$ in $M_{H}(\lambda)$ and a
basis $\{(\psi_{i},\varphi_{i})\}_{i=k}^{k+q-1}$ in
$M_{H}(\lambda^*)$.\\
 \indent {\bf Step 2.}~
Solve the following equations: find $\beta_{i}\in\mathbb{C}$,
$i=k,\cdots,k+q-1$ such that
\begin{eqnarray*}
\sum\limits_{i=k}^{k+q-1}\beta_{i}A((\psi_{i},\varphi_{i}),(\psi_{l},\varphi_{l}))=A((u_{H},\omega_{H}),(\psi_{l},\varphi_{l})),~~~l=k,\cdots,k+q-1.
\end{eqnarray*}
Then
$(u_{H}^{-},\omega_{H}^{-})=\sum_{i=k}^{k+q-1}\beta_{i}(\psi_{i},\varphi_{i})$ and
$(u_{H}^{*},\omega_{H}^{*})=(u_{H}^{-},\omega_{H}^{-})/\|(u_{H}^{-},\omega_{H}^{-})\|_A$.\\
\indent How to determine the multiplicity $q$ of $\lambda$ is a
difficult task in mathematics and has few achievements by now.
In practical computation, $q$ is determined by computational
experience. When $q$ cannot be   determined exactly, we also use
 $q_{\lambda_{H}}$ to replace $q$, where
$q_{\lambda_{H}}$ is  the multiplicity of $\lambda_{H}$ that can be
easily determined in practical computation. \\

\indent Our analysis makes use of the following lemma which is a
generalization of Lemma 9.1 in \cite{babuska} and Lemma 3.6 in
\cite{kolman}.
\begin{lemma}
Let $(\lambda,u,\omega)$ and $(\lambda^{*},u^{*},\omega^{*})$ be the
eigenpair of (\ref{s2.6}) and (\ref{s2.15}), respectively. Then,
$\forall (v,z),(v^{*},z^{*})\in \mathbf{H}$, $B((v,z),
(v^{*},z^{*}))\not=0$, the generalized Rayleigh quotient satisfies
\begin{eqnarray}\label{s4.4}
&&\frac{A((v,z),(v^{*},z^{*}))}{B((v,z),(v^{*},z^{*}))}-\lambda=\frac{A((v,z)-(u,\omega),(v^{*},z^{*})-(u^{*},\omega^{*}))}{B((v,z),(v^{*},z^{*}))}
\nonumber\\
&&~~~~~~-\lambda
\frac{B((v,z)-(u,\omega),(v^{*},z^{*})-(u^{*},\omega^{*}))}{B((v,z),(v^{*},z^{*}))}.
\end{eqnarray}
\end{lemma}
\begin{proof}
 From (\ref{s2.6}) and (\ref{s2.15}) we have
\begin{eqnarray}
&&A((v,z)-(u,\omega),(v^{*},z^{*})-(u^{*},\omega^{*}))-\lambda B((v,z)-(u,\omega),(v^{*},z^{*})-(u^{*},\omega^{*}))\nonumber\\
&&~~~=A((v,z),(v^{*},z^{*}))-A((v,z),(u^{*},\omega^{*}))\nonumber\\
&&~~~~~~-A((u,\omega),(v^{*},z^{*}))+A((u,\omega),(u^{*},\omega^{*}))\nonumber\\
&&~~~~~~-\lambda(B((v,z),(v^{*},z^{*}))-B((v,z),(u^{*},\omega^{*}))\nonumber\\
&&~~~~~~-B((u,\omega),(v^{*},z^{*}))+B((u,\omega),(u^{*},\omega^{*})))\nonumber\\
&&~~~=A((v,z),(v^{*},z^{*}))-\lambda B((v,z),(u^{*},\omega^{*}))\nonumber\\
&&~~~~~~-\lambda B((u,\omega),(v^{*},z^{*}))+\lambda B((u,\omega),(u^{*},\omega^{*}))\nonumber\\
&&~~~~~~-\lambda(B((v,z),(v^{*},z^{*}))-B((v,z),(u^{*},\omega^{*}))\nonumber\\
&&~~~~~~-B((u,\omega),(v^{*},z^{*}))+B((u,\omega),(u^{*},\omega^{*})))\nonumber\\
&&~~~=A((v,z),(v^{*},z^{*}))-\lambda B((v,z),(v^{*},z^{*})),\nonumber
\end{eqnarray}
dividing both sides by $B((v,z),(v^{*},z^{*}))$   we
obtain the desired conclusion.
\end{proof}

\begin{theorem}
Let $\lambda_{H}, (u_{H},\omega_{H}), (u_{H}^{*},\omega_{H}^{*}),
\lambda^{h}, (u^{h},\omega^{h}), (u^{h*},\omega^{h*})$ be the
numerical eigenpairs obtained by Scheme 4.1, and
$\lambda_{H}=\lambda_{j,H}$ $(j=k,k+1,\cdots,k+q-1)$.
 Assume that the ascent of   $\lambda$
are equal to $1$, and  that R(D) and $(C2)$ are valid, $n \in
W^{2-s,\infty}(D)$ ($s=0,1$). Then there exists  $(u,\omega)\in
R(E)$ and
 $(u^{*},\omega^{*})\in R(E^{*})$ such that when $H$ is
properly small there hold
\begin{eqnarray}\label{s4.5}
&&\|(u^{h},\omega^{h})-(u,\omega)\|_{\mathbf{H}}\lesssim
H^{r_{s}}\varepsilon_{H}(\lambda)
+\varepsilon_{h}(\lambda),\\\label{s4.6}
&&\|(u^{h*},\omega^{h*})-(u^{*},\omega^{*})\|_{\mathbf{H}}\lesssim
H^{r_{s}}\varepsilon_{H}^{*}(\lambda^{*})
+\varepsilon_{h}^{*}(\lambda),\\\label{s4.7} && \mid
\lambda^{h}-\lambda \mid \lesssim
\{H^{r_{s}}\varepsilon_{H}(\lambda) +\varepsilon_{h}(\lambda)\}
\{H^{r_{s}}\varepsilon_{H}^{*}(\lambda^{*})
+\varepsilon_{h}^{*}(\lambda^{*})\},~~~s=0,1.
\end{eqnarray}
\end{theorem}
 \begin{proof}
 Let $(u,\omega)\in R(E)$ such that $(u,\omega)-(u_{H},\omega_{H})$ and $\lambda-\lambda_{H}$ both satisfy Theorem
3.3 and Theorem 3.5. From (\ref{s2.12}) we get $(u,\omega)=\lambda T
(u,\omega)$, and from (\ref{s3.3}) and Step 2 of Scheme 4.1, we get
$ (u^{h},\omega^{h})=\lambda_{H}T_{h}(u_{H},\omega_{H}). $ From
(\ref{s3.3}), we have
\begin{eqnarray*}
\|\lambda_{H}T_{h}(u_{H},\omega_{H})-\lambda T_{h}
(u,\omega)\|_{\mathbf{H}}\lesssim
\|\lambda_{H}(u_{H},\omega_{H})-\lambda
(u,\omega)\|_{\mathbf{H}^{s}},
\end{eqnarray*}
and thus by using Theorem 3.3 and Theorem 3.5, we derive that
\begin{eqnarray*}
&&\|(u^{h},\omega^{h})-(u,\omega)\|_{\mathbf{H}}=\|\lambda_{H}T_{h}(u_{H},\omega_{H})-\lambda T (u,\omega)\|_{\mathbf{H}}\nonumber\\
&&~~~\leq \|\lambda_{H}T_{h}(u_{H},\omega_{H})-\lambda T_{h}
(u,\omega)\|_{\mathbf{H}} +\|\lambda T_{h} (u,\omega)-\lambda T
(u,\omega) \|_{\mathbf{H}}
\nonumber\\
&&~~~\lesssim \|\lambda_{H}(u_{H},\omega_{H})-\lambda
(u,\omega)\|_{\mathbf{H}^{s}} +|\lambda|\|T_{h} (u,\omega)-T
(u,\omega) \|_{\mathbf{H}}
\nonumber\\
&&~~~\lesssim H^{r_{s}}\varepsilon_{H}(\lambda)
+\varepsilon_{h}(\lambda),
\end{eqnarray*}
i.e., (\ref{s4.5}) holds.
 Similarly we can prove (\ref{s4.6}).\\
\indent From (\ref{s4.4}), we have
\begin{eqnarray}\label{s4.8}
&&\lambda^{h}-\lambda=\frac{A((u^{h},\omega^{h})-(u,\omega),(u^{h*},\omega^{h*})-(u^{*},\omega^{*}))}{B((u^{h},\omega^{h}),(u^{h*},\omega^{h*}))}\nonumber\\
&&~~~~~~-\lambda
\frac{B((u^{h},\omega^{h})-(u,\omega),(u^{h*},\omega^{h*})-(u^{*},\omega^{*}))}{B((u^{h},\omega^{h}),(u^{h*},\omega^{h*}))}.
\end{eqnarray}
Note that $(u_{H},\omega_{H})$  and $(u^{h},\omega^{h})$ just
approximate the same eigenfunction $(u,\omega)$,
$(u_{H}^{*},\omega_{H}^{*})$ and~$(u^{h*},\omega^{h*})$ approximate
the same adjoint eigenfunction $(u^{*},\omega^{*})$, and
$|B((u_{H},\omega_{H}), (u_{H}^*,\omega_{H}^{*}))|$ has a positive
lower bound uniformly with respect to $H$. From
\begin{eqnarray*}
&&B((u^{h},\omega^{h}),(u^{h*},\omega^{h*}))
=B((u^{h},\omega^{h}),(u^{h*},\omega^{h*}))-B((u,\omega),(u^{*},\omega^{*}))\\
&&~~~~~~+B((u,\omega),(u^{*},\omega^{*})) -B((u_{H},\omega_{H}),
(u_{H}^*,\omega_{H}^{*})) +B((u_{H},\omega_{H}),
(u_{H}^*,\omega_{H}^{*})),
\end{eqnarray*}
we know that $|B((u^{h},\omega^{h}),(u^{h*},\omega^{h*}))|$ has a
positive lower bound uniformly. Therefore from (\ref{s4.8}), we get
\begin{eqnarray}\label{s4.9}
|\lambda^{h}-\lambda|\lesssim
\|(u^{h},\omega^{h})-(u,\omega)\|_{\mathbf{H}}\|(u^{h*},\omega^{h*})-(u^{*},\omega^{*})\|_{\mathbf{H}}.
\end{eqnarray}
Substituting (\ref{s4.5}) and (\ref{s4.6}) into (\ref{s4.9}) yields
(\ref{s4.7}).
\end{proof}


\section{Numerical experiment}

\indent In this section, we will report  some numerical experiments
for the finite element discretization (\ref{s3.1}) and two grid
discretization scheme (Scheme 4.1)
to validate our theoretical results.\\
\indent We use MATLAB 2012a to solve (\ref{s1.1})-(\ref{s1.4}) on a
Lenovo G480 PC with 4G memory. Our program is implemented using the
package   iFEM \cite{chen}.\\
\indent Let $\{\xi_i\}_{i=1}^{N_h}$ be a basis of $S^h$ and
$u_h=\sum_{i=1}^{N_h}u_i\xi_i,\omega_h=\sum_{i=1}^{N_h}\omega_i\xi_i$.
A similar definition can be made for $u^*_h$ and $\omega^*_h$.
Denote $\overrightarrow{u}=(u_1,\cdots,u_{N_h})^T$ and
$\overrightarrow{\omega}=(\omega_1,\cdots,\omega_{N_h})^T$.
Similarly  $\overrightarrow{u*}$ and $\overrightarrow{\omega*}$ can
be defined from $u^*_h$ and $\omega^*_h$. To describe our algorithm,
we specify
the following matrices in  the discrete case.\\
\begin{center} \footnotesize
\begin{tabular}{lllll}\hline
Matrix&Dimension&Definition\\\hline
$A_h$&$N_h\times N_h$&$a_{l,i}=\int_D \frac{1}{n-1}\Delta\xi_i\Delta\xi_ldx$\\
$B_h$&$N_h\times N_h$&$b_{l,i}=-\int_D \{\frac{1}{n-1}\xi_i\Delta
\xi_l +
 \Delta \xi_i\frac{1}{n-1}\xi_l-
\nabla \xi_i\cdot\nabla\xi_l
\}dx$\\
$C_h$&$N_h\times N_h$&$c_{l,i}=-\int_D \frac{n}{n-1} \xi_i \xi_ldx$\\
$D_h$&$N_h\times N_h$&$d_{l,i}=\int_D  \xi_i \xi_ldx$\\
\hline
\end{tabular}
\end{center}
\noindent when $n\in W^{1,\infty}(D)$, $b_{l,i}=\int_D
\{\nabla(\frac{1}{n-1}\xi_i)\cdot\nabla \xi_l +
 \nabla \xi_i\cdot \nabla(\frac{n}{n-1}\xi_l) \}dx$.\\
Then (\ref{s3.1}) and (\ref{s3.11}) can be written as the
generalized eigenvalue problems
\begin{eqnarray} \label{s5.1}
\left(
\begin{array}{lcr}
A_h&0\\
0&D_h
\end{array}
\right) \left (
\begin{array}{lcr}
\overrightarrow{u}\\
\overrightarrow{\omega}
\end{array}
\right)=\lambda_h  \left(
\begin{array}{lcr}
B_h&C_h\\
D_h&0
\end{array}
\right) \left (
\begin{array}{lcr}
\overrightarrow{u}\\
\overrightarrow{\omega}
\end{array}
\right),
\end{eqnarray}
and
\begin{eqnarray} \label{s5.2}
\left(
\begin{array}{lcr}
A_h&0\\
0&D_h
\end{array}
\right) \left (
\begin{array}{lcr}
\overrightarrow{\overline{u}*}\\
\overrightarrow{\overline{\omega}*}
\end{array}
\right)=\lambda_h  \left(
\begin{array}{lcr}
B_h&D_h\\
C_h&0
\end{array}
\right) \left (
\begin{array}{lcr}
\overrightarrow{\overline{u}*}\\
\overrightarrow{\overline{\omega}*}
\end{array}
\right).
\end{eqnarray}

\indent Note that in (\ref{s5.1}), $A_{h}$ is a positive definite
Hermitian matrix and $D_{h}$ can be equivalently replaced by the
identity matrix $I_h$, which will lead to two sparser coefficient
matrices with a good properties. Thus the computation of the
eigenpairs for (\ref{s5.1}) is efficient.
Similarly, the matrices   for Scheme 4.1 can be given but are  omitted in this paper.\\
\indent For   convenience, we use the following notations in our tables and figures:\\
\indent $k_{j,h}=\sqrt{\lambda_{j,h}}$: The $jth$ eigenvalue
obtained by
(\ref{s3.1}) on $\pi_{h}$.\\
\indent $k_{j,H}=\sqrt{\lambda_{j,H}}$: The $jth$ eigenvalue
obtained by
(\ref{s3.1}) on $\pi_{H}$.\\
 \indent $ k_{j}^{h} =\sqrt{\lambda_{j}^h}$: The $jth$ eigenvalue obtained by Scheme
4.1.\\
\indent ---:     Failure of computation due to running out of memory.

\subsection{Model problem on the unit square}

\indent We first consider the case when $D$ is the unit square
$[0,1]^2$ and the index of refraction $n=16$ or $n=8+x_{1}-x_{2}$.
We use BFS element to compute the problem, and the numerical results
are shown in Tables 5.1-5.2. We also depict the error curves for the
numerical eigenvalues (see Figure 5.1).

\indent According to   regularity theory, we know $r_{0}=2$ and $u,
\omega\in H^{4}(D)$. When the ascent $\alpha=1$: according to
(\ref{s3.19}) and (\ref{s3.15}), the convergence order of the
eigenvalue approximation $k_{j,h}$ is four; according to
(\ref{s4.7}), when $h\gtrsim H^{2}$, the convergence order of the
$k_{j}^h$ is also four, i.e.,
\begin{eqnarray}\label{s5.3}
|k_{j}-k_{j,h}|\lesssim h^{4},~~~ |k_{j}-k_{j}^{h}|\lesssim h^{4}.
\end{eqnarray}

\indent It is seen from Figure 5.1 that the convergence order of the
numerical eigenvalues on the unit square is four, which
coincides with the theoretical result (\ref{s5.3}).\\
\indent In addition, Tables 5.1-5.2 show that the numerical
eigenvalues obtained by BFS elements give a good approximation; it
is worthwhile noticing that   two grid discretization scheme can achieve the
same convergence order as solving the eigenvalue problem by BFS
element directly. Moreover, we find that the two grid discretization scheme can be
performed on  finer meshes so that we can obtain more accurate
numerical eigenvalues. Therefore, the  two grid discretization scheme for
solving transmission eigenvalue problem is more efficient.\\

\begin{table}
\caption{The eigenvalues obtained by BFS element on the unit square,
$n=16$.}
\begin{center} \footnotesize
\begin{tabular}{ccccccc}\hline
$j$&$H$&$h$&$k_{j,H}$&$k_{j}^{h}$&$k_{j,h}$\\\hline
1&  $\frac{\sqrt2}{8}$& $\frac{\sqrt2}{32}$&    1.8800518272&   1.8795932933&   1.8795931085&\\
1&  $\frac{\sqrt2}{16}$&    $\frac{\sqrt2}{128}$&   1.8796216444&   1.8795911813&   1.8795911812&\\
1&  $\frac{\sqrt2}{32}$&    $\frac{\sqrt2}{256}$&   1.8795931085&   1.8795911697&   1.8795911747&\\
2&  $\frac{\sqrt2}{8}$& $\frac{\sqrt2}{32}$&    2.4462555154&   2.4442475976&   2.4442447101&\\
2&  $\frac{\sqrt2}{16}$&    $\frac{\sqrt2}{128}$&   2.4443713201&   2.4442361446&   2.4442361333&\\
2&  $\frac{\sqrt2}{32}$&    $\frac{\sqrt2}{256}$&   2.4442447101&   2.4442361002&---\\
3&  $\frac{\sqrt2}{8}$& $\frac{\sqrt2}{32}$&    2.4462555154&   2.4442475976&   2.4442447101&\\
3&  $\frac{\sqrt2}{16}$&    $\frac{\sqrt2}{128}$&   2.4443713201&   2.4442361446&   2.4442361334&\\
3&  $\frac{\sqrt2}{32}$&    $\frac{\sqrt2}{256}$&   2.4442447101&   2.4442361002& ---   \\
4&  $\frac{\sqrt2}{8}$& $\frac{\sqrt2}{32}$&    2.8681931483&   2.8664515120&   2.8664469634&\\
4&  $\frac{\sqrt2}{16}$&    $\frac{\sqrt2}{128}$&   2.8665606968&   2.8664391607&   2.8664391408&\\
4&  $\frac{\sqrt2}{32}$&    $\frac{\sqrt2}{256}$&   2.8664469634&
2.8664391111& ---
\\\hline
\end{tabular}
\end{center}
\end{table}

\begin{table}
\caption{The eigenvalues obtained by BFS element on the unit square,
$n=8+x_1-x_2$.}
\begin{center} \footnotesize
\begin{tabular}{ccccccc}\hline
$j$&$H$&$h$&$k_{j,H}$&$k_{j}^{h}$&$k_{j,h}$\\\hline
1&  $\frac{\sqrt2}{8}$& $\frac{\sqrt2}{32}$&    2.8234457937&   2.8221946996&   2.8221945051\\
1&  $\frac{\sqrt2}{16}$&    $\frac{\sqrt2}{128}$&   2.8222709846&   2.8221893629&   2.8221893619\\
1&  $\frac{\sqrt2}{32}$&    $\frac{\sqrt2}{256}$&   2.8221945051&   2.8221893480&   ---\\
2&  $\frac{\sqrt2}{8}$& $\frac{\sqrt2}{32}$&    3.5424522436&   3.5387180040&   3.5387126105\\
2&  $\frac{\sqrt2}{16}$&    $\frac{\sqrt2}{128}$&   3.5389469722&   3.5386967795&   3.5386967579\\
2&  $\frac{\sqrt2}{32}$&    $\frac{\sqrt2}{256}$&   3.5387126105&   3.5386967008&---\\
5,6&    $\frac{\sqrt2}{8}$& $\frac{\sqrt2}{32}$&    4.4971031374 &   4.4964665266 & 4.4965591247 \\
&&&  $\pm$.8770188489$i$&     $\pm$0.8715534026$i$&  $\pm$0.8715053132$i$\\
5,6&    $\frac{\sqrt2}{32}$&    $\frac{\sqrt2}{128}$&   4.4965591247 &  4.4965519816&     4.4965519832 \\
&&&  $\pm$0.8715053132$i$&   $\pm$0.8714818735$i$&     $\pm$0.8714818728$i$\\
5,6&    $\frac{\sqrt2}{32}$&    $\frac{\sqrt2}{256}$&   4.4965591247 &  4.4965519531 &  ---\\
&&&  $\pm$0.8715053132$i$&   $\pm$0.871481788$i$~~& ---
\\\hline
\end{tabular}
\end{center}
\end{table}

\subsection{Model problem on the L-shaped domain}

We consider the case when $D$=$(-1,1)^2\backslash ([0,1)\times (-1,
0])$ is the L-shaped domain and the index of refraction $n=16$ or
$n=8+x_{1}-x_{2}$.   The numerical results obtained by BFS element
are shown in Tables 5.3-5.4 and Figure 5.2.\\
\indent Figure 5.2 indicates that  on the L-shaped domain, the
convergence order of $k_{1,h},k_{2,h}$ is around 1.3 and 2.3 respectively for
$n=16$, and
 the convergence order of $k_{1,h},k_{5,h}$ is around 1.4 for $n=8+x_1-x_2$.
  This fact suggests that the four eigenfunctions on the non-convex domain do
have singularities to different degrees.

\begin{table}
\caption{The eigenvalues obtained by BFS element on the L-shaped
domain, $n=16$.}
\begin{center} \footnotesize
\begin{tabular}{ccccccc}\hline
$j$&$H$&$h$&$k_{j,H}$&$k_{j}^{h}$&$k_{j,h}$\\\hline
1&  $\frac{\sqrt2}{8}$& $\frac{\sqrt2}{32}$&    1.4850653844&   1.4781249432&   1.4780403370\\
1&  $\frac{\sqrt2}{16}$&    $\frac{\sqrt2}{64}$&    1.4802422297&   1.4770298927&   1.4770116105\\
1&  $\frac{\sqrt2}{16}$&    $\frac{\sqrt2}{128}$&   1.4802422297&   1.4765529659&   ---\\
2&  $\frac{\sqrt2}{8}$& $\frac{\sqrt2}{32}$&    1.5705634174&   1.5697720130&   1.5697716222\\
2&  $\frac{\sqrt2}{16}$&    $\frac{\sqrt2}{64}$&    1.5699010557&   1.5697385580&   1.5697385335\\
2&  $\frac{\sqrt2}{16}$&    $\frac{\sqrt2}{128}$&   1.5699010557&   1.5697293878&   ---\\
3&  $\frac{\sqrt2}{8}$& $\frac{\sqrt2}{32}$&    1.7078128918&   1.7055170229&   1.7055794458\\
3&  $\frac{\sqrt2}{16}$&    $\frac{\sqrt2}{64}$&    1.7061981331&   1.7052820968&   1.7052949545\\
3&  $\frac{\sqrt2}{16}$&    $\frac{\sqrt2}{128}$&   1.7061981331&   1.7051443044&   ---\\
4&  $\frac{\sqrt2}{8}$& $\frac{\sqrt2}{32}$&    1.7834680505&   1.7831208381&   1.7831208523\\
4&  $\frac{\sqrt2}{16}$&    $\frac{\sqrt2}{64}$&    1.7831489498&   1.7831171004&   1.7831171097\\
4&  $\frac{\sqrt2}{16}$&    $\frac{\sqrt2}{128}$&   1.7831489498&   1.7831163182&   ---
\\\hline
\end{tabular}
\end{center}
\end{table}

\begin{table}
\caption{The eigenvalues obtained by BFS element on the L-shaped
domain, $n=8+x_1-x_2$.}
\begin{center} \footnotesize
\begin{tabular}{ccccccc}\hline
$j$&$H$&$h$&$k_{j,H}$&$k_{j}^{h}$&$k_{j,h}$\\\hline
1&  $\frac{\sqrt2}{8}$& $\frac{\sqrt2}{32}$&    2.3127184233&   2.3045393229&   2.3043808707\\
1&  $\frac{\sqrt2}{16}$&    $\frac{\sqrt2}{64}$&    2.3069612717&   2.3032153953&   2.3031811119\\
1&  $\frac{\sqrt2}{16}$&    $\frac{\sqrt2}{128}$&   2.3069612717&   2.3026641395&   ---\\
2&  $\frac{\sqrt2}{8}$& $\frac{\sqrt2}{32}$&    2.3974892428&   2.3957871378&   2.3957863810\\
2&  $\frac{\sqrt2}{16}$&    $\frac{\sqrt2}{64}$&    2.3960567236&   2.3957182671&   2.3957182148\\
2&  $\frac{\sqrt2}{16}$&    $\frac{\sqrt2}{128}$&   2.3960567236&   2.3956994585&   ---\\
5,6&    $\frac{\sqrt2}{8}$& $\frac{\sqrt2}{32}$&    2.9287346709 &  2.9252416365 &  2.9257099237 \\
&&&  $\pm$0.5743458101$i$&   $\pm$0.5664545248$i$&   $\pm$0.5664307881$i$\\
5,6&    $\frac{\sqrt2}{16}$&    $\frac{\sqrt2}{64}$&    2.9272401495 &   2.9248351944&  2.9249318006\\
&&&  $\pm$0.5686045702$i$&    $\pm$0.5654537885$i$&  $\pm$0.5654487200$i$\\
5,6&    $\frac{\sqrt2}{16}$&    $\frac{\sqrt2}{128}$&   2.9272401495 &  2.9244335090&   ---\\
&&&  $\pm$0.5686045702$i$&   $\pm$0.5649994678$i$&---
\\\hline
\end{tabular}
\end{center}
\end{table}

 \begin{figure}
\includegraphics[width=0.5\textwidth]{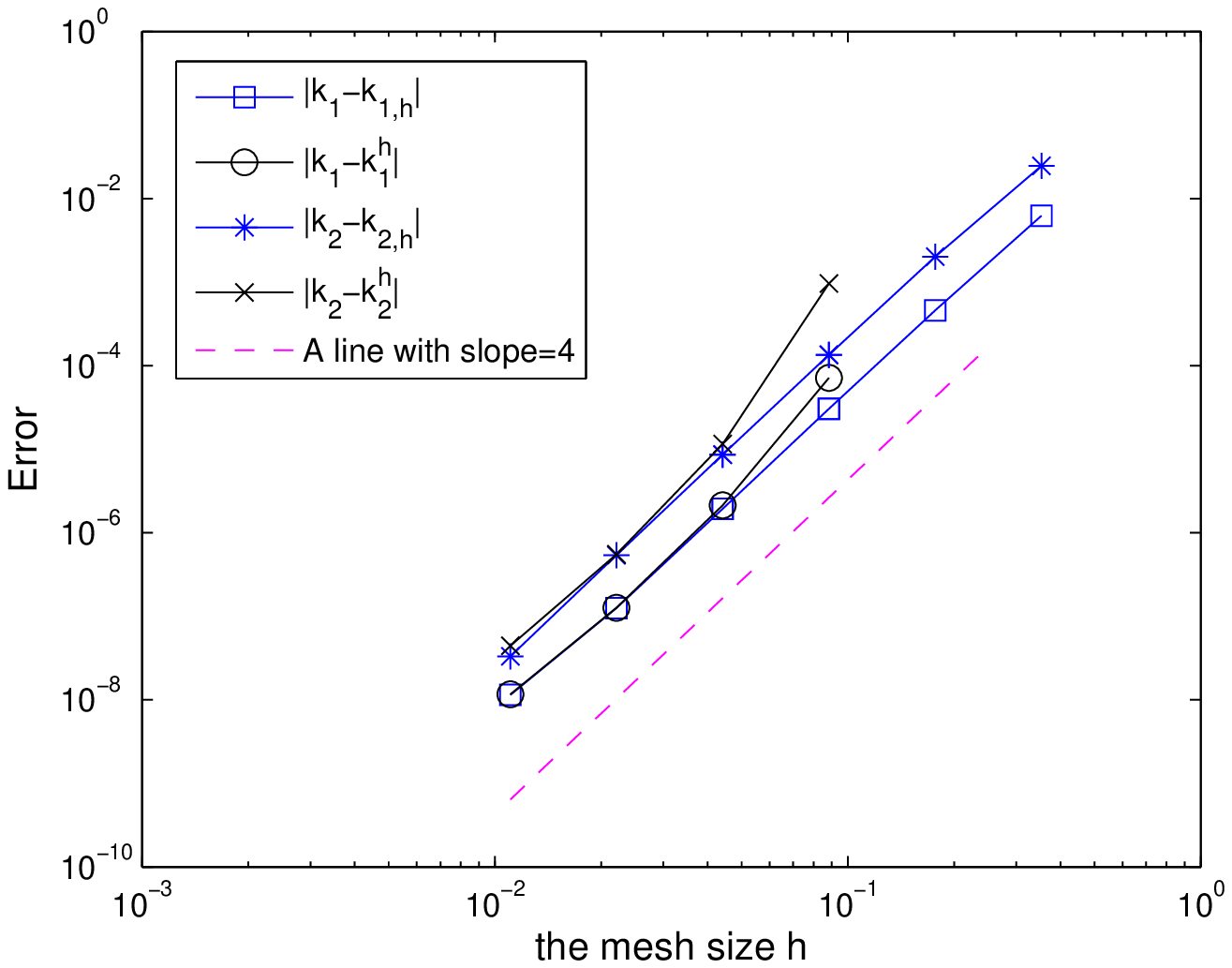}
 \includegraphics[width=0.5\textwidth]{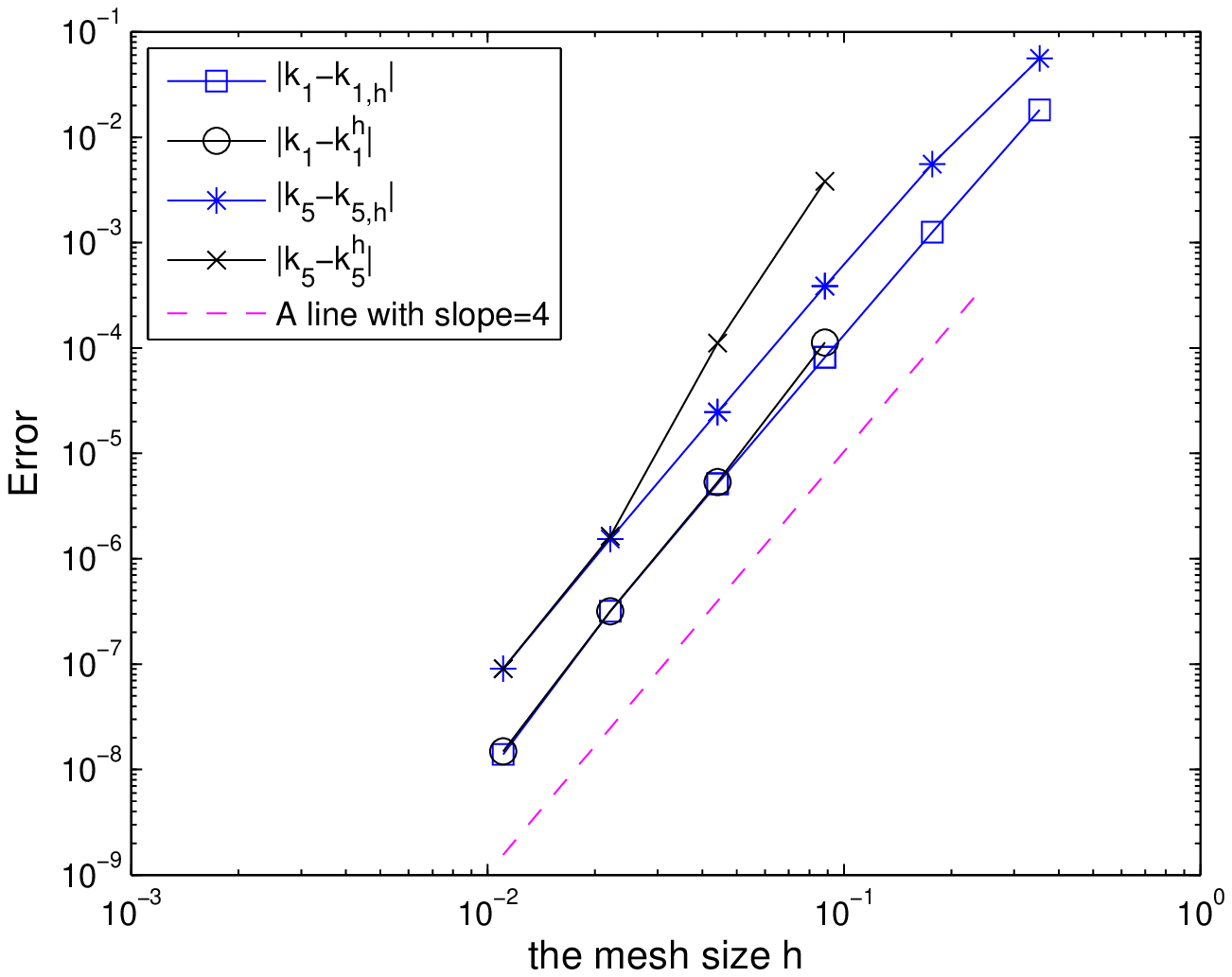}
\caption{{ Error curves on the unit square for $k_{1},k_{2}$  with
$n=16$ (left), and for $k_{1},k_{5}$  with $n=8+x_1-x_2$ (right).}}
 \end{figure}
  \begin{figure}
  \includegraphics[width=0.5\textwidth]{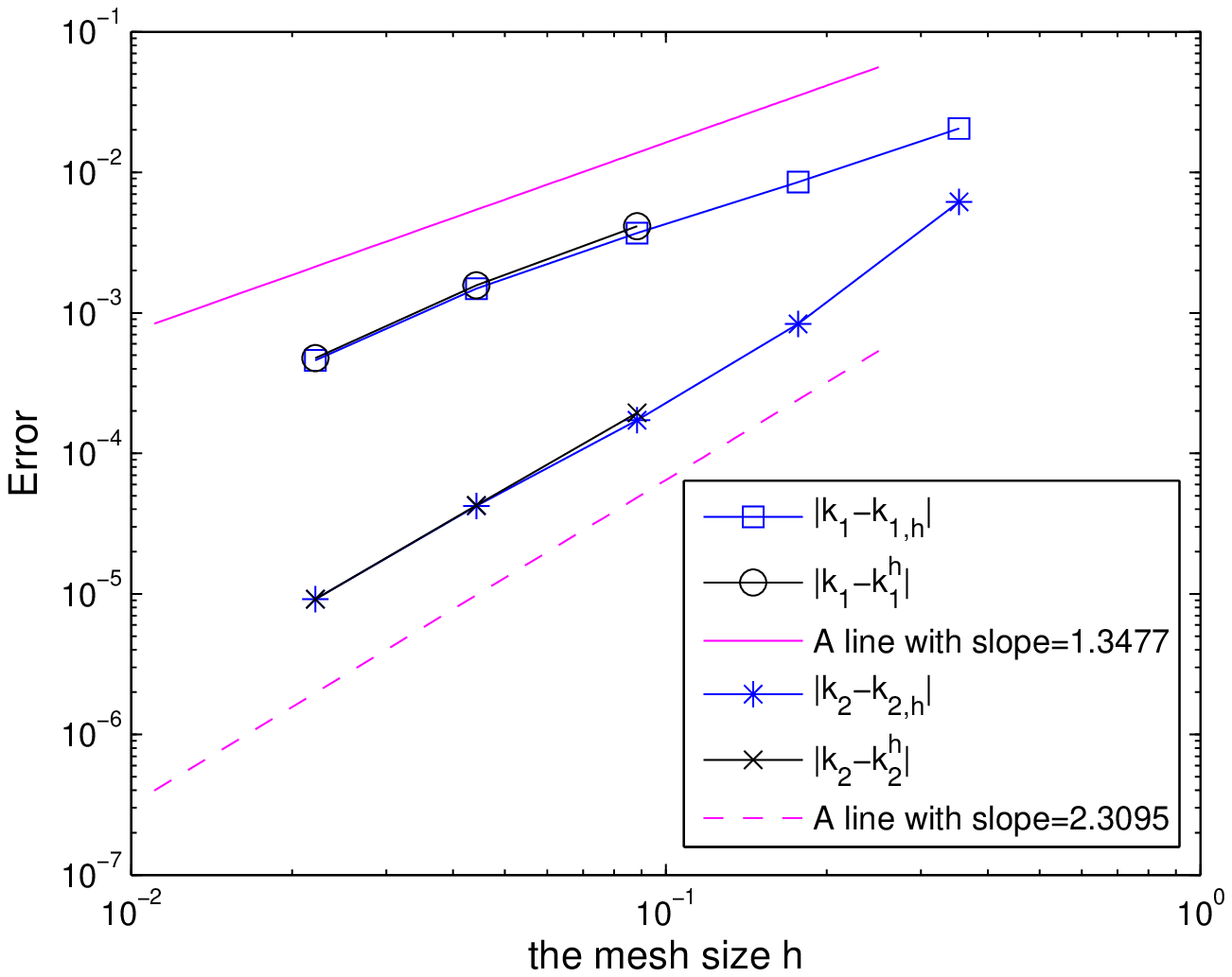}
 \includegraphics[width=0.5\textwidth]{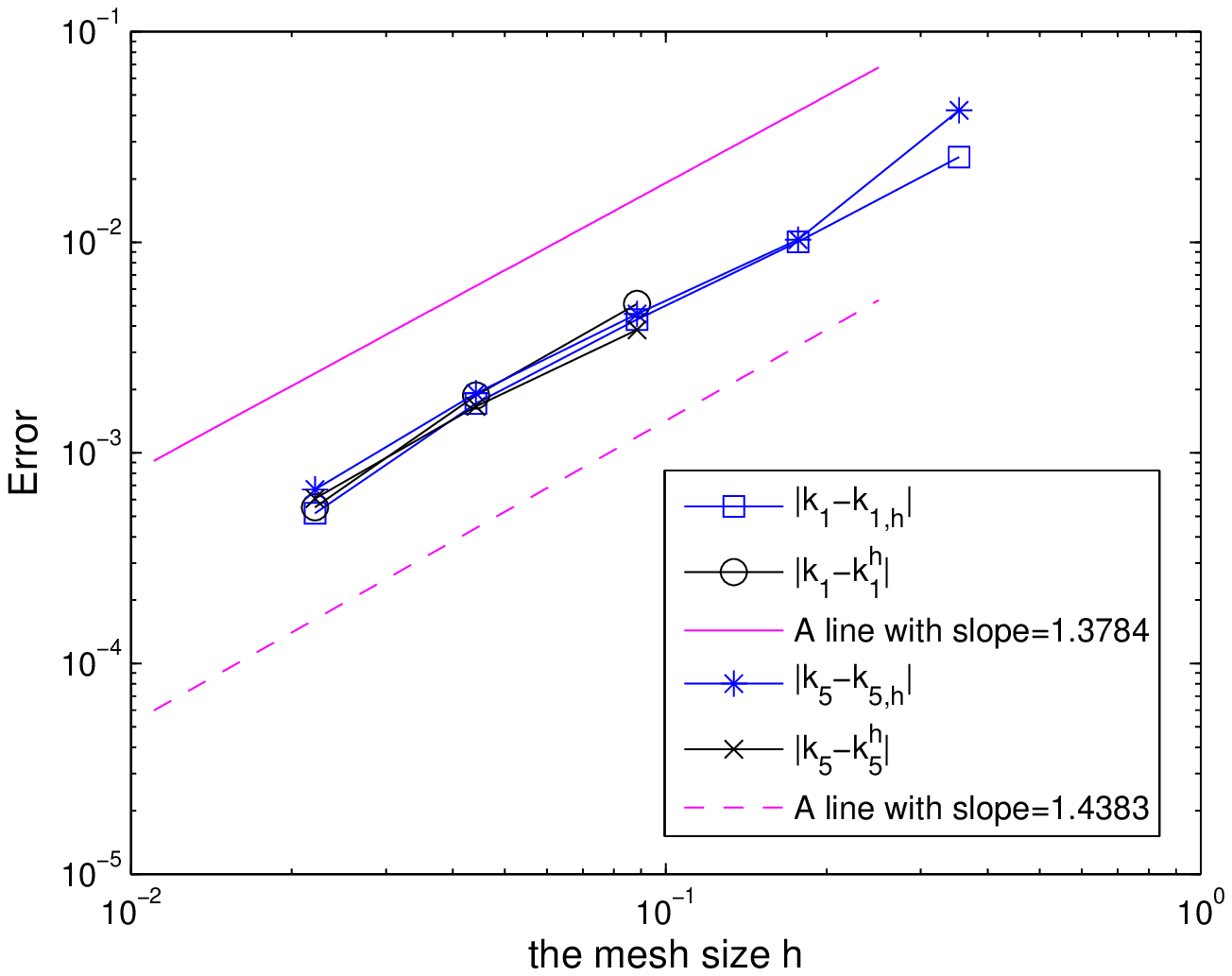}
\caption{{ Error curves on the L-shaped domain for $k_{1},k_{2}$  with $n=16$ (left), and for $k_{1},k_{5}$  with $n=8+x_1-x_2$ (right).}}
 \end{figure}

\subsection{Model problem on the circular domain}
We also investigate the case of $n$ being piecewise constant for a
disk $D$ of the radius 1. Let
   $n(x)=n_1$ for $x\in~D_1$ and $n(x)=n_2$ for $x\in~D\backslash \overline{D_1}$,
   $D_1$ being an inner disk of the radius $r_1<1$. For this disk domain we generate a triangular mesh with
    $h\approx\frac{1}{40}$ and number of degrees of freedom  101040 and move the nodes outside and nearest the inner disk onto the inner circle.
     And we use the Argyris element to solve (\ref{s1.1})-(\ref{s1.4}) on the mesh and the  numerical eigenvalues associated with different
     $n$ and $r_1$ are shown in Table 5.5.
The analytic eigenvalues can be referred to Tables 2 and 3 in
\cite{gintides}.
\begin{table}
\caption{The eigenvalues obtained by Argyris element on the circle
of radius 1 for piecewise constant $n$.}
\begin{center} \footnotesize
\begin{tabular}{lll}\hline
$n_1,n_2,r_1$&\{$k_{1,h},k_{2,h},k_{3,h},k_{4,h},k_{5,h},k_{6,h},k_{7,h},k_{8,h},k_{9,h},k_{10,h},k_{11,h}$\}\\\hline
13 5 0.5  &1.4975,1.7353,1.7357,2.1709,2.1721,2.3570$\pm$0.4881i,2.6995,2.6999\\
~2 4 0.5  &2.4349$\pm$0.6969i,3.5812$\pm$0.5731i,3.5824$\pm$0.5730i,3.9137,3.9196,4.1160\\
~5 8 0.6 &1.7891,2.2512,2.2520,2.5574$\pm$0.3780i,2.6681,2.6717,3.0376,3.0388\\
~5 3 0.6  &2.2996$\pm$0.7437i,2.8876,2.8885,3.1861,3.2953,3.2982\\
10 8 0.7  &1.3732,1.7322,1.7328,2.1101,2.1121,2.3665$\pm$0.4401i,2.4958,2.4964\\
~2 4 0.7  &2.4280$\pm$0.6299i,3.7582$\pm$0.6135i,3.7593$\pm$0.6136i,4.9689$\pm$0.4384i,    4.9751$\pm$0.4395i,5.0515\\
13 11 0.8 &1.1499,1.4899,1.4903,1.8261,1.8278,2.1570,2.1575\\
~3 ~6 0.8 &2.2223$\pm$0.4796i,3.0286,3.0297,3.6711,3.6748,3.8012\\
~6 13 0.9 &1.5929,2.0271,2.0280,2.4687$\pm$0.2643i,2.5142,2.5171,3.0034,3.0043\\
~6 ~2 0.9 &2.0222,2.4190,2.4192,2.7401$\pm$0.4468i,2.9309,2.9317
\\\hline
\end{tabular}
\end{center}
\end{table}

%

\end{document}